\documentclass[a4paper, 11pt, headings = small, abstract]{scrartcl}
\usepackage{RGDM_preamble}

\bibliography{literature}
\title{\fontsize{16}{19} \selectfont Statistical guarantees for denoising reflected diffusion models}
\author{Asbjørn Holk\thanks{Aarhus University, Department of Mathematics, Ny Munkegade 118, 8000 Aarhus C, Denmark. \newline Email: \href{mailto:a.holk@math.au.dk}{a.holk@math.au.dk}} \and Claudia Strauch\thanks{Heidelberg University, Institute for Mathematics, Im Neuenheimer Feld 205, 69120 Heidelberg, Germany. \newline Email: \href{mailto:strauch@math.uni-heidelberg.de}{strauch@math.uni-heidelberg.de}} \and Lukas Trottner\thanks{University of Stuttgart, Department of Mathematics, Wankelstraße 5, 70563 Stuttgart, Germany. \newline Email: \href{mailto:lukas.trottner@isa.uni-stuttgart.de}{lukas.trottner@isa.uni-stuttgart.de}}}
\date{\today}

\begin{document}
\maketitle

\begin{abstract}
In recent years, denoising diffusion models have become a crucial area of research due to their abundance in the rapidly expanding field of generative AI. While recent statistical advances have delivered explanations for the generation ability of idealised denoising diffusion models for high-dimensional target data, implementations introduce thresholding procedures for the generating process to overcome issues arising from the unbounded state space of such models. This mismatch between theoretical design and implementation of diffusion models has been addressed empirically by using a \emph{reflected} diffusion process as the driver of noise instead. In this paper, we study statistical guarantees of these denoising reflected diffusion models. {\color{black}In particular, under Sobolev smoothness assumptions, we establish rates of convergence in total variation which, up to a polylogarithmic factor, match the minimax lower bound.} Our main contributions include the statistical analysis of this novel class of denoising reflected diffusion models and a refined score approximation method in both time and space, leveraging spectral decomposition and rigorous neural network analysis.
\end{abstract}

\section{Introduction}\label{sec:intro}
Deep generative models (DGMs) are a broad class of models that train deep neural networks to generate synthetic samples from a target distribution representing a given training data set. Examples that have shown tremendous empirical success in the last decade include the classes of Generative Adversarial Networks (GANs) \citep{goodfellow14,nowo16,arjovsky17,zhao17}, Variational Autoencoders  \citep{kingma14} and normalising flows \citep{papa21,rezende15,tabak13}. Most recently, dynamic generative methods, which work by learning the reverse dynamics of a forward noising process that gradually evolves the original data distribution into a simple and easy-to-sample-from distribution, have gained much traction in the machine learning community. The most prominent class of such models are Denoising Diffusion Models (DDMs) in discrete and continuous time \citep{song19,ho20,song21,sohl-d15}, which add Gaussian noise to the data and for which the backward dynamics is determined by the gradient of the log likelihood of the forward marginals, also known as the score. The score is intractable as it depends on the unknown initial data distribution and therefore needs to be learned, which is why these models are also referred to as score-based generative models. 

These models have been adapted and improved for enhanced efficiency and performance in numerous applications. However, the statistical theory for standard DGMs is still under active development with many basic questions left unanswered. The statistical theory for GANs is perhaps the most mature among DGMs, with significant efforts in the recent past \citep{puchkin24,tang23,schreuder21,chen22,chae22,liang21,stephan24,stephan23} to build a minimax theory that, among other things, draws on deeper connections to optimal transport theory \citep{villani09,chewi24}.

\paragraph{Related work}
Starting only recently with the seminal paper \cite{oko23}, the study of statistical convergence guarantees for standard denoising diffusion models has experienced rapid growth in the past two years.  {\color{black}For denoising diffusion models in $\R^d$ that transform data into noise via a forward Ornstein--Uhlenbeck process, \cite{oko23} prove total variation and Wasserstein-1 rates that match the minimax lower bound over Besov classes up to log factors for the total variation and arbitrarily small polynomial loss for the Wasserstein-1 distance}. Their assumptions on the initial distribution with density $p_0$ can be summarised by three key components: 
\begin{enumerate}[label = (\roman*)]
\item $p_0$ is compactly supported on the $d$-dimensional hypercube $[-1,1]^d$; 
\item $p_0$ is bounded away from zero on its support;
\item $p_0$ has Besov smoothness of order $s$ away from the support boundary (where $s$ is allowed to be sufficiently small to not necessarily imply continuity of $p_0$) and is infinitely differentiable close to the boundary.
\end{enumerate}
These assumptions provide a natural starting point from a statistical perspective, since they allow one to easily control certain quantities in the approximation analysis that would otherwise obscure the central mathematical ideas. However, these assumptions have practical limitations since, in many empirical applications, data distributions tend to be multimodal and extremely high-dimensional. Consequently, the resulting optimal rates, derived in terms of the ambient dimension $d$, often do not align with the empirical success of diffusion models in generating high-quality samples in such complex scenarios. This gap has motivated recent studies to extend these results to cases where data distributions are constrained to lower-dimensional structures, aligning with the widely held \emph{manifold hypothesis}, which suggests that popular high-dimensional training data sets are supported on lower-dimensional manifolds. \cite{chen23b}, \cite{tang24} and \cite{azangulov24} have expanded the analysis to such structured settings, refining the  rates in Wasserstein-1 distance and providing theoretical foundations for understanding the adaptivity of DDMs to low-dimensional structures.

\paragraph{Our contribution}
Our work offers a different type of extension: we provide statistical convergence guarantees for \textit{denoising reflected diffusion models} (DRDMs). This class of generative models was introduced and empirically implemented in \cite{lou23}, motivated by the observation that, in practice, the implementation of the generation process of standard DDMs relies significantly on the incorporation of thresholding procedures to prevent the backward process, which has unbounded state space, from exiting the permitted range of target samples (say, e.g., tensors containing bounded RGB values of pixels in an image). Since such thresholding procedures have no theoretical justification in the design of unconstrained diffusion models, \cite{lou23} suggest using reflected diffusion processes on bounded domains $D$ as the forward noising process and demonstrate competitive performance to state-of-the-art  models with comparable training and inference times. {\color{black}Specifically, \cite{lou23} implement a time-changed reflected Brownian motion as forward model, which yields a uniform ergodic distribution on the bounded domain $\overline{D}$ that is used as initialisation of the backward generative process with learned score obtained by denoising score matching, see below.} We also refer to \cite{fishman23} and \cite{fishman23b} for related work on constrained diffusion models. 
{\color{black}
To help the reader quickly identify the main technical novelties and to place our results in the context of the existing literature, let us briefly summarise core contributions of this paper:
\begin{itemize}
    \item We derive an explicit upper bound on the expected total variation distance between the data distribution $p_0$ and the distribution induced by the DRDM with estimated score, with a resulting convergence rate which matches the well-known minimax lower bounds over Sobolev classes up to logarithmic factors.
    \item We derive a spectral characterisation of the time-dependent score function for reflected diffusions generated by self-adjoint operators with Neumann boundary conditions, yielding an explicit representation amenable to statistical analysis.
    \item We construct a calibrated neural network approximation scheme for these spectral score representations, combining truncation of eigenexpansions with space-time interpolation to balance approximation accuracy and model complexity.
\item We analyse how the semi-explicit nature of transition densities in the reflected setting impacts score estimation and show how the resulting approximation and estimation errors can be controlled despite the absence of closed-form Gaussian expressions.
\end{itemize}
Having outlined these contributions, we now describe the construction and analysis of the reflected denoising dynamics in more detail.} Time-reversal of the forward noising process motivates the formulation of the backward generative model, which under suitable regularity assumptions is again given by a reflected diffusion whose drift is determined by the score function
\[\sco^\circ(x,t) \coloneq \nabla \log p_t(x), \quad p_t(x) \coloneq \int_D q_t(y,x)\, \PP(X_0 \in \diff{y}),\]
where $q_t(x,y)$ are the transition densities of the reflected forward diffusion and $\PP(X_0 \in \cdot)$ is the data distribution. Since the latter is unknown,  the score must be learned from a given i.i.d.\ data sample $(X_{0,i})_{i = 1,\ldots,n}$, which, similarly to the unconstrained diffusion case, is implemented by minimising an empirical version of the denoising score matching loss 
\[\int_{\underline{T}}^{\overline{T}} \int_{D^2} \lvert \sco(y,t) - \nabla_y \log q_t(x,y) \rvert^2 q_t(x,y) \,\PP(X_0 \in \diff{x}) \diff{y} \diff{t},\]
in a class of feedforward ReLU neural networks $\mathcal{S} \ni \sco$, where $\overline{T} - \underline{T} \in (0, \overline{T}]$ is the  stopping time of the backward reflected diffusion. By the equivalence of denoising and explicit score matching \citep{vincent11}, it therefore becomes crucial to precisely calibrate the neural network class in terms of the number of observations $n$, the dimension $d$ and the smoothness $s$ of the data to balance the explicit score approximation error
\[\min_{\sco \in \mathcal{S}} \int_{\underline{T}}^{\overline{T}} \int_D \lvert\sco(x,t) - \nabla \log p_t(x) \rvert^2 p_t(x) \diff{x} \diff{t}\]
and the complexity of the class $\mathcal{S}$, measured in terms of its covering number w.r.t.\ the supremum norm. 
While the approximation analysis for standard diffusion models with Ornstein--Uhlenbeck forward noising process \citep{oko23,chen23b,tang24,azangulov24} can rely on the explicitly known Gaussian transition densities $q_t(x,y)$, this situation dramatically changes in the reflected model, where the transition densities are at best semi-explicit and need to be numerically approximated. We choose a forward reflected diffusion with self-adjoint weighted Laplacian $\nabla \cdot f\nabla$ subject to a Neumann boundary condition as a generator, which provides us with a space-time spectral decomposition of the transition densities of the form
\begin{equation}\label{eq:trans_dens}
q_t(x,y) = \sum_{j=0}^\infty \mathrm{e}^{-t\lambda_j} e_j(x)e_j(y), \quad t > 0, x,y \in D,
\end{equation}
for eigenpairs $(\lambda_j,e_j)_{j \in \N_0}$ of $-\nabla \cdot f\nabla$ with known growth behaviour. 
The statistical estimation of the conductivity $f$ based on such reflected diffusion data has been recently studied for low- and high-frequency observations in \cite{nickl23} and \cite{hoffmann22}, respectively. In the low-frequency statistical setting, the computational challenges arising  from the semi-explicit nature of \eqref{eq:trans_dens} have been subsequently analysed in \cite{giordano24}. {\color{black}Let us note that, for the specific choice of $f\equiv 1/2$, the generator is given by $\tfrac{1}{2} \Delta$ and thus yields a reflected Brownian motion as a forward model. Up to a time change, we therefore cover the reflected diffusion models implemented in \cite{lou23}. Independently of our choice for the heat conductivity $f\colon \R^d \to [f_{\min},\infty) \subset (0,\infty)$, the forward noising process has a uniform ergodic distribution from which we sample to initialise the backward generating process. The conductivity $f(x)$ controls the speed at which the forward process diffuses around $x$, with larger values pushing the process faster towards equilibrium, but also increasing the oscillation of the eigenfunctions $e_j$ around $x$. This implies a tradeoff between the convergence speed of the forward model and the approximation difficulty of the backward model. For some potential $V \colon \R^d \to \R$ we could also consider the generator $\mathcal{A}_V \coloneq \mathrm{e}^V \nabla \cdot \mathrm{e}^{-V} f \nabla$ with Neumann boundary conditions instead, resulting in a normally reflected forward diffusion with stationary distribution $\mu_V \propto \mathrm{e}^{-V}$ and a spectral decomposition as in \eqref{eq:trans_dens} with respect to the dominating measure $\mu_V(y) \diff{y}$, see \cite{nickl23}. For the particular choice of $f = 1/2$, this yields a reflected (overdamped) Langevin diffusion process. Our general approximation approach, which is described below, could be adapted to this scenario. However, in the context of generative modelling, we are primarily interested in a stationary forward distribution that is easy to sample from. As long as the domain $D$ is sufficiently nice, this is clearly satisfied for the uniform ergodic distribution in the model considered here, but it is generally difficult for space-dependent choices of $V$, which require sampling from $\mu_V \propto \mathrm{e}^{-V}$ via MCMC methods, which are expensive in the dimension $d$.  From a practical perspective, it is therefore quite natural to consider reflected diffusion models with uniform stationary distribution.}

To get a grasp on the neural network approximation of the score obtained from \eqref{eq:trans_dens}, we  impose assumptions on the data distribution $\PP(X_0 \in \cdot)$, which in the technical setting of \cite{nickl23} provide a natural analogue to the three essential conditions assumed for the data distribution in the statistical analysis of unconstrained diffusion models alluded to above. More precisely, let $H^s_c(D)$ be the class of compactly supported Sobolev functions of order $s$ on $D$, where we extend  any function $\varphi \in H^s_c(D)$ to $\overline{D}$ by setting $\varphi\vert_{\partial D} = 0$. We assume that 
\begin{enumerate}[label = ($\mathcal{H}$0), ref = ($\mathcal{H}$0)]
\item \label{ass:init} there exist a nonnegative function $\tilde{p}_0 \in H^s_c(D)$, for $s \in \N \cap (d/2,\infty)$, and $\alpha > 0$ such that $\PP(X_0 \in \cdot)$ has a density $p_0\colon \overline{D} \to [\alpha,\infty)$ given by $p_0 = \tilde{p}_0 + \alpha$.
\end{enumerate}
This simplifying assumption provides a strictly positive lower bound $\alpha > 0$ on the data distribution, while at the same time avoiding boundary issues associated with reflection thanks to $p_0 \in H^s_c(D)/\R$ and guaranteeing Hölder continuity of $p_0$ by the Sobolev smoothness condition $s > d/2$. 
{\color{black} These conditions are directly comparable to the assumptions made on the data density $p_0$ in \cite{oko23}:
\begin{itemize}
\item in their unconstrained model, they allow the noising and denoising processes to leave the support $[-1,1]^d$ and assume a lower bound $p_0\vert_{[-1,1]^d} \geq \alpha > 0$. We enforce the support constraint by restricting the generative model to $\overline{D} = \supp p_0$ and therefore assume $p_0 = p_0\vert_{\overline{D}} \geq \alpha > 0$.
\item \cite{oko23} assumes Besov smoothness $p_0 \in B^{s}_{p,q}$ for $s > (1/p -1/2) \vee 0$, with the additional restriction that $p_0$ is arbitrarily smooth near the boundary. Since $H^s = B^s_{2,2}$, our Sobolev assumption on $p_0$ is more restrictive, and our requirement that $s > d/2$ implies continuity of $p_0$. For the particular case $p=q=2$, however, \cite{oko23} impose no restrictions on the smoothness parameter $s > 0$. We enforce arbitrary smoothness close to the boundary in a specific way by modelling $p_0$ as the shift of a compactly supported function on the open domain $D$.
\end{itemize}
} Given the preceding discussion, our focus in this paper is therefore \emph{not} to optimise our initial data assumptions with regard to what is observed in practice for popular data sets. We rather opt for a set of conditions that shares common principles with previous studies for standard DDMs, but allows us to highlight the central mathematical ideas for score approximation in the reflected setting and to develop a fairly compact approximation theory that we regard as a versatile starting point for future statistical investigations.

Given \ref{ass:init}, the score can be written as 
\[\sco(x,t) = \nabla \log p_t(x) = \frac{\sum_{j=0}^\infty \mathrm{e}^{-t\lambda_j}  \langle p_0, e_j\rangle_{L^2} \nabla e_j(x)}{\sum_{j=0}^\infty \mathrm{e}^{-t\lambda_j} \langle p_0, e_j \rangle_{L^2} e_j(x)}, \quad x \in D, t > 0.\]
Our neural network approximation strategy is then broken down into the following steps: 
\begin{enumerate} 
\item truncate the series representation of $p_t(x)$ at $j=N$ to obtain an approximation $h_N(x,t)$ and corresponding truncated score $\nabla \log h_N(x,t)$, where $N$ needs to be appropriately chosen depending on $n,s$ and $d$;
\item for an appropriately chosen discrete set of time points $\{t_i\}$, use the spatial smoothness of $h_N(x, t_i)$ induced by the Sobolev smoothness of $p_0$ to obtain an efficient neural network approximation of $h_N(\cdot, t_i)$, based on general approximation results from \cite{suzuki19};
\item approximate the space-time function $h_N(x,t)$ by constructing a neural network approximation of the time interpolation of the neural networks from Step 2., where the interpolation degree is adapted to the parameters $N,s$ and $d$.
\end{enumerate}
All steps require a careful calibration of the approximation parameters and accordingly the neural network sizes, which is made possible by a precise analysis of the different levels of numerical errors induced by our stepwise approach. These techniques may be of independent interest since they are applicable to generic spectral decompositions of semigroups associated to self-adjoint Markov generators. In  particular, this is potentially relevant for extending the statistical analysis to the unifying class of \textit{denoising Markov models} that has been recently introduced in \cite{benton_shi24}, see also \cite{ren25}: {\color{black} following our previous discussion of generalised reflected forward models with self-adjoint generator $\mathcal{A}_V = \mathrm{e}^V \nabla \cdot f\mathrm{e}^{-V} \nabla$ and associated spectral decomposition of the transition density $q_t(x,y)$ relative to the invariant density $\mu_V \propto \mathrm{e}^{-V}$, for a general self-adjoint Markov generator $\mathcal{A}$ with invariant distribution $\mu$ and  eigendecomposition $(\lambda_j,e_j)_{j \in \N_0}$ in $L^2(\mu)$, the unknown forward marginals that are needed to express the backward generator in denoising Markov models are decomposed as 
\[p_t(x) = \sum_{j=0}^\infty \mathrm{e}^{-t\lambda_j} \langle p_0, e_j \rangle_{L^2(\mu)} e_j(x).\]
Thus, under appropriate assumptions on $p_0$ and controls on the eigenvalues and eigenfunctions, as well as with detailed knowledge of the invariant distribution $\mu$, a neural-network approximation of the space-time function $p_t(x)$ and functionals thereof appears plausible, using the general strategy introduced in this paper.
}

This strategy allows us to determine an explicit calibration of the neural network class $\mathcal{S}$, the forward terminal time $\overline{T}$ and the backward early stopping time $\overline{T} - \underline{T}$, such that for the empirical denoising score loss minimiser $\hat{\mathfrak{s}}_n$ in this class of neural networks, we obtain the convergence rate 
\[\E\big[\TV(p_0,\cev{p}{}^{\hat{\mathfrak{s}}_n}_{\overline{T}- \underline{T}})\big] \lesssim n^{-\frac{s}{2s+d}} (\log n)^{3} (\log \log n)^{1/2}.\]
Here, $\cev{p}{}^{\hat{\mathfrak{s}}_n}_t$ denotes the density at time $t$ of the backward generating reflected diffusion that has drift term determined by $\hat{\mathfrak{s}}_n$ and is started in the uniform distribution on $D$. {\color{black}This establishes an expected total variation convergence rate (up to small $\log$-factors) for denoising reflected diffusion models that matches the well-known minimax lower bound over Sobolev classes \citep{yang99}.}

\paragraph{Organisation of the paper} 
The paper is structured as follows. In Section \ref{sec:setting}, we provide the necessary technical background on denoising diffusion models. Here, we briefly outline the fundamentals underlying standard unconstrained DDMs, before providing a precise mathematical framework for denoising reflected diffusion models. Section \ref{sec:main} introduces the exact specification of the score estimator via denoising score matching and the associated generative model. We then present our main result, Theorem \ref{theo:main}, with the remainder of the paper dedicated to its proof. The proof preparation proceeds along three sections. In Section \ref{subsec:decomp}, we present the basic error decomposition of the expected total variation risk and establish bounds on the first two sources of error, arising from early stopping and the initiation of the backwards generating process in the invariant uniform distribution on $\overline{D}$. The score matching error as the last component of the error decomposition is discussed in Section \ref{subsec:score_match}, before in Section \ref{subsec:approx}, we construct the neural network approximation of the score that allows us to optimally bound the score matching error. Section \ref{susbsec:proof_main} is then dedicated to proving Theorem \ref{theo:main}, based on the results of the previous subsections. 
In the concluding Section \ref{sec:discussion}, we discuss our findings and highlight promising directions for future research, motivated by the insights and limitations of this paper.

\paragraph{Notation} 
For an open set $D$ and $s \in \N$, we denote by $H^s(D)$ the Sobolev space of functions having weak partial derivatives up to order $s$ in $L^2(D)$. We denote the inner product and corresponding norm on $L^2(D)$ by $\langle \cdot , \cdot \rangle_{L^2(D)}$ and $\lVert \cdot \rVert_{L^2(D)}$, or simply $\langle \cdot , \cdot \rangle_{L^2}$ and $\lVert \cdot \rVert_{L^2}$ if the domain $D$ is fixed and there is no room for confusion. $\lvert \cdot \rvert$ denotes the Euclidean norm on $\R^d$ for any $d \in \mathbb{N}$, and for a function $f\colon \mathcal{X} \to \R^d$ for some space $\mathcal{X}$, $\lVert f \rVert_{\mathcal{X}} = \sup_{x \in \mathcal{X}} \lvert f(x) \rvert$ is its supremum norm.  $\mathcal{C}^s(D)$ denotes the space of functions with continuous partial derivatives of order $s$ in $\overline{D}$, and we let $\mathcal{C}^\infty(D) = \bigcap_{s \in \N} \mathcal{C}^s(D)$. For $\beta \in (0,1]$, $\mathcal{C}^{0,\beta}(D)$ is the space of $\beta$-Hölder continuous functions on $D$. For $1\leq p,q \leq \infty$ and $\alpha > 0$, $\mathcal{B}^\alpha_{p,q}(D)$ denote the usual Besov spaces on $D$, cf.\ \cite{triebel10} for details.

\section{Theoretical background}\label{sec:setting}
Before introducing our framework for DRDMs, we first recall the basic ideas of unconstrained diffusion models. 
Based on observing a finite number of samples corresponding to an \textit{unknown} distribution $p_0$ on $\R^d$, DDMs provide an iterative generative algorithm to create new samples that approximately match the target distribution $p_0$. The general idea is to find a stochastic process that perturbs $p_0$ to a new distribution $p_T$ in such a way that 1) $p_T$ or a good approximation thereof is easy to sample from, and 2) the perturbation is reversible in the sense that we know how to simulate the time-reversed process. In an idealised setting where we have access to the exact specifications of the backward dynamics, samples from $p_0$ can be generated exactly by first sampling from $p_T$ and then running the backward process. However, these dynamics must naturally adapt to the information contained in the unknown $p_0$, so the true backward dynamics must be estimated from the data. 

In the framework of DDMs, this perturbation is done via an SDE, i.e., for some fixed time $T>0$ and suitable drift $b\colon[0, T]\times\R^d\to\R^d$ and diffusion coefficient $\sigma\colon[0, T]\times\R^d\to\R^{d\times d}$, we consider the forward model
\[
	\diff X_t
	=b(t, X_t)\diff{t}+\sigma(t, X_t)\diff{W_t},\quad t\in[0, T], X_0\sim p_0,
\]
where $W=(W_t)_{t\in[0, T]}$ is a standard $d$-dimensional Brownian motion.
Under sufficient regularity conditions \citep{anderson82,haussmann86}, the forward model has a solution $X=(X_t)_{t\in[0, T]}$ with marginal densities $(p_t)_{t \in [0,T]}$ such that the time-reversed process $\cev{X}_t = X_{T-t}$, $t \in [0,T]$, solves
\begin{equation}
	\label{forward_backward_SDE}
	\diff\cev{X}_t
	=-\overline{b}(T-t, \cev{X}_t)\diff{t}+\sigma(T-t, \cev{X}_t)\diff{\overline{W}_t},\quad t\in[0, T], \cev{X}_0\sim p_T,
\end{equation}
for some Brownian motion $(\overline{W}_t)_{t \in [0,T]}$ and drift $\overline{b}\colon [0,T] \times \R^d \to \R^d$  given by 
\[
	\overline{b}_i(t, x)=b_i(t, x)-\frac{1}{p_t(x)}\sum_{j, k=1}^{d}\frac{\partial}{\partial x_j}\big[p_t(x)\sigma_{ik}(t, x)\sigma_{jk}(t, x)\big],\quad i=1,\ldots, d.
\]
Thus, the time-reversed process solves a  time-inhomogeneous SDE, with drift $-\overline{b}(T-\cdot,\cdot)$ and diffusion coefficient $\sigma(T-\cdot,\cdot)$. 
In many practical implementations as well as in the statistical studies \cite{oko23,chen23b,tang24,azangulov24}, the diffusion coefficient is set to $\sigma(t,x) = \gamma(t) \mathbb{I}_d$ for some scalar function $\gamma$, which implies that the forward model is given by a (possibly time-inhomogeneous) Ornstein--Uhlenbeck process with explicit transition densities and the backward drift becomes
\[
	\overline{b}(t, x)
	=b(t, x)-\gamma^2(t)\nabla\log p_t(x),
\]
where $\nabla \log p_t$ is referred to as the \textit{score} of the forward model. 
Substituting this into \eqref{forward_backward_SDE}, we obtain the dynamics
\[
	\diff\cev{X}_t
	=\big(-b(T-t, \cev{X}_t)+\gamma^2(T-t)\nabla\log p_{T-t}(\cev{X}_t)\big)\diff t+\gamma(T-t)\diff \overline{W}_t\quad t\in[0, T], \overline{X}_0\sim p_T,
\]
of the backward process. Then, when $t\rightarrow T$, the density of $\cev{X}_t$ approaches $p_0$, so that simulating the reverse process generates new data samples corresponding to the target $p_0$. 
While we are free to choose the coefficients of our forward process (i.e., $b$ and $\sigma$), the score function $\nabla \log p_t$ depends on $p_0$ and hence needs to be estimated from the data, which is referred to as \emph{score matching}.

\paragraph{Reflected generative diffusion models}
RGDMs follow the same generative principle, but constrain both forward and backward dynamics to a bounded domain. To avoid some technicalities, let us assume that $D\subseteq\R^d$ is an open, connected and bounded set with $\mathcal{C}^\infty$ boundary $\partial D$. We consider the reflected time-homogeneous forward model 
\[\diff{X_t}  = b(X_t) \diff{t} + \sigma(X_t) \diff{W_t} + \nu(X_t) \diff{\ell_t}, \quad X_0 \in \overline{D},\] 
with smooth and bounded coefficients $b\colon \overline{D} \to \R^d$, $\sigma\colon \overline{D} \to \R^{d \times d}$ and conormal reflection determined by 
\[\nu(x) \coloneqq \frac{1}{2}a(x)n(x) = \frac{1}{2}\sum_{i=1}^d \langle \sigma_{\cdot,i}(x), n(x) \rangle \sigma_{\cdot,i}(x) > 0, \quad x \in \partial D,\]
where $a \coloneq \sigma\sigma^\top$ is assumed to be uniformly elliptic. 
Here, $n$ is the inward unit normal vector at the boundary $\partial D$, and the process $(\ell_t)_{t \geq 0}$ is the local time at $\partial D$, which is a nondecreasing continuous process of bounded variation that increases only when the solution $X$ hits the boundary, i.e., $\ell_t = \int_0^t \one_{\partial D}(X_s) \diff{\ell_s}$ almost surely. Because of the conormal reflection, the boundary local time can equivalently be characterised by 
\begin{equation}\label{eq:occ}
\ell_t = \lim_{\varepsilon \to 0} \frac{1}{\varepsilon} \int_0^t \one_{(\partial D)_{\varepsilon}}(X_s) \diff{s},
\end{equation}
where $(\partial D)_{\varepsilon} \coloneqq \{x \in \R^d: \operatorname{dist}(x,\partial D) \leq \varepsilon \}$, and the limit holds both in $L^2$ and almost surely, uniformly on $[0,T]$, cf.\ \citet[Proposition 1.3]{cattiaux88}. 
The boundary reflection process $L_t \coloneqq \int_0^t \varphi(X_s) \diff{\ell_s}$ reflects $X$ in a conormal direction whenever it hits the boundary $\partial D$, thus constraining the state space of the diffusion to the compact set $\overline{D}$. 
The process $X$ is a time-homogeneous Markov process with transition semigroup $(Q_t)_{t \geq 0}$ determined by transition densities $(q_t)_{t \geq 0}$, given as the fundamental solutions of the PDE with conormal Neumann boundary condition,
\begin{align*} 
\begin{cases}
\frac{\partial}{\partial t} u(x,t) = \mathcal{A} u(x,t), &(x,t) \in D \times (0,T] ,\\
\frac{\partial}{\partial \nu} u(x,t) = 0, &(x,t) \in \partial D \times (0,T],
\end{cases}
\end{align*} 
where $\mathcal{A}$ is the second-order differential operator given by 
\[\mathcal{A} = \sum_{i=1}^d b_i(x) \partial_{x_i} + \frac{1}{2}\sum_{i,j=1}^d a_{i,j}(x) \partial_{x_i} \partial_{x_j}.\]
We denote the density of the forward process at time $t$ by $p_t$, i.e., 
\[p_t(x) \diff{x} = \PP(X_t \in \diff{x}\mid X_0 \sim p_0) = \int_{\overline{D}} p_0(y) \,Q_{t}(y,\diff{x}) \diff{y} = \int_{\overline{D}} p_0(y) q_{t}(y,x) \diff{y} \diff{x}, \quad x \in \overline{D}.\]
Similarly to the unconstrained model, \citet[Theorem 2.5]{cattiaux88} shows that time-reversion of the reflected diffusion process yields a time-inhomogeneous reflected diffusion process, whose drift is reminiscent of the unconstrained model. 
More precisely, letting $\cev{X}{}^T = (\cev{X}_t)_{t \in [0,T]}$, $\cev{X}_t = X_{T-t}$ and $\overline{\ell}{}^T = (\overline{\ell}_t)_{t \in [0,T]}$, $\overline{\ell}_t = \ell_{T} - \ell_{T-t}$, $t \in [0,T]$, there exists a Brownian motion $\overline{W}{}^T = (\overline{W}_t)_{t \in [0,T]}$ w.r.t.\ an enlargement of the natural filtration generated by $\cev{X}{}^T$ such that $\cev{X}{}^T$ solves 
\begin{equation}\label{eq:rsde_reflect}
\diff{\cev{X}_t} = -\overline{b}( \cev{X}_t)\diff{t}+\sigma(\cev{X}_t)\diff{\overline{W}_t} + \nu(\cev{X}_t) \diff{\overline{\ell}_t}, \quad \cev{X}_0 \sim p_T,
\end{equation}
on $[0,T)$, where $\overline{b}\colon [0,T] \times \overline{D} \to \R^d$ is given by 
\[
	\overline{b}_i(t, x)=b_i(t, x)-\frac{1}{p_t(x)}\sum_{j, k=1}^{d}\frac{\partial}{\partial x_j}\big[p_t(x)\sigma_{ik}(t, x)\sigma_{jk}(t, x)\big],\quad i=1,\ldots,d.
\]
Note that, by definition, $\overline{\ell}{}^T$ is nondecreasing, has bounded variation and satisfies $\overline{\ell}_t = \int_0^t \one_{\partial D}(\cev{X}_s) \diff{\overline{\ell}_s}$ for $t \in [0,T]$, i.e., $\overline{\ell}{}^T$ is the local time at the boundary of the backward process $\cev{X}{}^T$.

A fundamental requirement for diffusion generative modeling is a precise understanding of the limiting behaviour of the forward process to evaluate the required run time of the forward process for the backward initialisation to be a sufficiently good  approximation of the true terminal forward distribution $p_T$.  To this end, we choose $b = \nabla f$ and $\sigma = \sqrt{2f} \mathbb{I}_{d \times d}$ for some potential $f\colon\R^d \to [f_{\min},\infty) \subset (0,\infty)$. {\color{black}To avoid technicalities, we assume that $f \in \mathcal{C}^\infty(\overline{D})$ and that the smooth and bounded domain $D$ is also convex. This enables us to remain within the technical framework of \cite{nickl23}, which provides useful technical results required for our analysis.} With our choice of coefficients, the time-homogeneous forward dynamics are described by the divergence form $L^2$-generator 
\begin{equation}\label{eq:generator}
\mathcal{A} = \nabla \cdot f \nabla = \langle \nabla f, \nabla \cdot\rangle  + f \Delta
\end{equation}
with core 
\[H^1_\nu(D) \coloneqq \{\varphi \in H^1(D) : \partial \varphi/\partial \nu = 0 \text{ on } \partial D\}\] 
for $\nu = fn$, corresponding to the constrained SDE
\begin{equation}\label{eq:rsde_heat}
\diff{X_t} = \nabla f(X_t) \diff{t} + \sqrt{2f(X_t)} \diff{W_t} + \nu(X_t)\diff{\ell_t}.
\end{equation}
Note that the ellipticity condition $f \geq f_{\min} > 0$ implies that the conormal Neumann boundary condition $\langle\nu(x), \nabla \varphi(x) \rangle = \tfrac{\partial \varphi}{\partial \nu}(x) = 0$ for $x \in \partial D$ is equivalent to a normal Neumann boundary condition, i.e., $H^1_\nu(D) = H^1_n(D)$. 
Thus, both the reflected forward and backward SDEs exhibit normal reflection at the boundary, and \eqref{eq:rsde_heat} induces a space-dependent scaling of local time that ensures that the occupation limit \eqref{eq:occ} holds true. Moreover, the specific choice $f \equiv 1/2$ yields a normally reflected Brownian motion.

By the divergence theorem, it follows that the invariant distribution of the forward Markov process $X$ is the easy-to-sample-from uniform distribution on $\overline{D}$, i.e., $\mu = \tfrac{\operatorname{Leb}\vert_{\overline{D}}}{\operatorname{Leb}(\overline{D})}$. Furthermore, there exist orthonormal eigenpairs $(\lambda_j, e_j)_{j \geq 0}$ of the nonnegative operator $-\nabla \cdot f\nabla$ satisfying $0= \lambda_0 < \lambda_1 \leq \lambda_2 \leq \cdots $ and obeying the Weyl asymptotics $\lambda_j \asymp j^{2/d}$ and $e_0 = \tfrac{1}{\operatorname{Leb}(D)^{1/2}} \one$ and $(e_j)_{j \geq 1} \subset H^1_\nu(D) \cap L^2_0(D)$ such that
\[q_t(x,y) = \sum_{j \geq 0} \mathrm{e}^{-t\lambda_j} e_j(x)e_j(y), \quad x,y \in D.\]
See \citet[Section 3]{nickl23} for a detailed discussion of these properties. 
Because $f \in \mathcal{C}^\infty(\overline{D})$, \citet[Corollary 1]{nickl23} yields the bounds
\begin{align*} 
\lVert e_j \rVert_{H^k} &\lesssim \lambda_j^{k/2} \asymp j^{k/d}, \quad j \geq 1,\\ 
\lVert e_j \rVert_\infty &\lesssim j^{\tau}, \quad \text{ for any } \tau > 1/2.
\end{align*}
This implies the smoothing property that, for any bounded initial density $p_0$, 
\[\lVert p_t \rVert_{H^k} \lesssim \lVert p_0 \rVert_\infty \sum_{j \geq 0} \mathrm{e}^{-t \lambda_j} \lVert e_j \rVert_\infty \lVert e_j \rVert_{H^k} \lesssim \lVert p_0 \rVert_\infty \mathrm{e}^{-t j^{2/d}} j^{\tau + k/d} < \infty, \quad t > 0,  \]
for arbitrary $\tau > 1/2$. By the Sobolev imbedding theorem, we therefore have $p_t \in \mathcal{C}^\infty(D)$, regardless of the smoothness properties of $p_0$, and we can identify the weak derivatives of $p_t$ with its classical derivatives.
The SDE \eqref{eq:rsde_reflect} governing the backward dynamics becomes
\[
\diff{\cev{X}_t} = \big(\nabla f(\cev{X}_t) + 2f(\cev{X}_t) \nabla\log p_{T-t}(\cev{X}_t) \big)\diff{t} + \sqrt{2f(\cev{X}_t)} \diff{\overline{W}_t} + \nu(\cev{X}_t) \diff{\overline{\ell}_t},
\]
with initialisation $\cev{X}_0 \sim p_T$. Thus, even though the diffusion coefficient is state-dependent, the particular interplay between drift and diffusion coefficient ensures that the backward drift is fully determined by the forward drift and the score $(x,t) \mapsto \nabla \log p_t(x)$, as in the unconstrained Ornstein--Uhlenbeck forward model with state-independent diffusion coefficient. 
By the spectral decomposition of the transition densities, the score is explicitly given by 
\begin{equation}\label{eq:score_spectral}
\nabla \log p_t(x) = \frac{\sum_{j \geq 0} \mathrm{e}^{-t\lambda_j} \langle p_0,e_j\rangle_{L^2} \nabla e_j(x)}{\sum_{j \geq 0} \mathrm{e}^{-t\lambda_j} \langle p_0,e_j\rangle_{L^2}   e_j(x)}, \quad x \in D, t > 0,
\end{equation}
which will be instrumental in analysing the score approximation properties of neural networks underlying the algorithm described in the next section.

\paragraph{Neural network classes} 
We will construct an estimator of the score via minimising the denoising score matching error in an appropriate class of neural networks. For doing so, we introduce  parameterised classes of neural networks with ReLU (Rectified Linear Unit) activation function. 
In particular, for any $b, x\in\R^m$, let
\[
	\sigma_b(x)
	=\mat{\sigma(x_1-b_1) \\ \sigma(x_2-b_2) \\ \vdots \\ \sigma(x_m-b_m)},\quad \sigma(y)=y\vee0,
\]
and denote for $L\in\N$, $W\in\N^{L+2}$, $S\in\N$ and $B>0$ by $\Phi(L, W, S, B)$ the class of neural networks with depth (i.e., number of hidden layers) $L$, layer widths (including input and output layers) $W$, sparsity constraint $S$, and norm constraint $B$.
We thus consider functions of the form
\[
	\varphi(x)
	=A_L\sigma_{b_L}A_{L-1}\sigma_{b_{L-1}}\cdots A_1\sigma_{b_1}A_0x,
\]
where $A_i\in\R^{W_{i+1}\times W_{i}},b_i\in\R^{W_{i+1}}$ for $i=0,\ldots,L$ (to ease notation, we always set $b_0=0$), and where there are at most a total of $S$ non-zero entries of the $A_i$'s and $b_i$'s and all entries are numerically at most $B$.
In an abuse of notation, we denote $\sigma_{0}$ simply by $\sigma$.
This can be written succinctly as
\[
	\Phi(L, W, S, B)
	\coloneqq\left\{
	\begin{aligned}
		&A_L\sigma_{b_L}A_{L-1}\sigma_{b_{L-1}}\cdots A_1\sigma_{b_1}A_0\mid A_i\in\R^{W_{i+1}\times W_i}, b_i\in\R^{W_{i+1}}, \\
		&\sum_{i=0}^L(\n{A_i}_0+\n{b_i}_0)\le S,\max_{i\in\{0, \ldots, L\}}(\n{A_i}_\infty\vee\n{b_i}_\infty)\le B
	\end{aligned}
	\right\}.
\]

\section{Generative modelling with reflected diffusions}\label{sec:main}
Denote the true score by $\sco^\circ(x,t) \coloneqq \nabla \log p_t(x)$, and assume we are given data samples $(X_{0,i})_{i \in [n]} \overset{\text{i.i.d.}}{\sim} p_0$. For a hypothesis class $\mathcal{S}$ of neural networks with ReLU activation function, to be exactly calibrated later,  and $\sco \in \mathcal{S} \cup \{\sco^\circ\}$, we define 
\begin{equation}\label{eq:lossfct}
\begin{split}
    L_{\sco}(x) &\coloneqq \int_{\underline{T}}^{\overline{T}} \int_{\overline{D}} \lvert \sco(y,t) - \nabla_y \log q_t(x,y) \rvert^2 q_t(x,y) \diff{y} \diff{t}\\
    &= \E\Big[\int_{\underline{T}}^{\overline{T}} \lvert \sco(X_t,t) - \nabla_y \log q_t(x,X_t) \rvert^2 \mid X_0 = x \Big],
\end{split}
\end{equation}
where $\overline{T}$ is the terminal runtime of the reflected forward process and $\underline{T} \in (0, \overline{T})$ is such that we run the reflected backward process,  which is initialised with distribution $\mathcal{U}(D)$,  until $\overline{T} - \underline{T}$. We then denote the empirical score matching loss associated to $\sco$ by
\[\hat{L}_{\sco,n} \coloneqq \frac{1}{n} \sum_{i=1}^n L_{\sco}(X_{0,i}),\]
and we define the empirical score minimiser by
\begin{equation}\label{eq:score_est}
\hat{\mathfrak{s}}_n \coloneqq \argmin_{\mathfrak{s} \in \mathcal{S}} \hat{L}_{\mathfrak{s},n}.
\end{equation}
We let $\overline{X}^{\sco}$ be a solution of the reflected SDE 
\begin{equation}\label{eq:rev_score_approx}
\begin{split}
\diff{\overline{X}^{\mathfrak{s}}_t} &= \big(\nabla f(\overline{X}^{\mathfrak{s}}_t) + 2f(\overline{X}^{\mathfrak{s}}_t) s(\overline{X}^{\mathfrak{s}},t) \big)\diff{t} + \sqrt{2f(\overline{X}^{\mathfrak{s}}_t)} \diff{\overline{W}_t} + \nu(\overline{X}^{\mathfrak{s}}_t) \diff{\overline{\ell}_t}, \quad t \in [0, \overline{T} - \underline{T}],\\
\overline{X}^{\mathfrak{s}}_0 &\sim \mathcal{U}(\overline{D}),
\end{split}
\end{equation}
for some Brownian motion $(\overline{W}_t)_{t \in [0, \overline{T} - \underline{T}]}$ and local time $(\overline{\ell}_t)_{t \in [0, \overline{T} - \underline{T}]}$ at the boundary $\partial D$, and we denote its density at time $t$ by $\cev{p}{}^{\mathfrak{s}}_t$. Here, the initialisation $X^{\hat{\mathfrak{s}_n}}_0 \sim \mathcal{U}(\overline{D})$ and the Brownian motion $\overline{W}$ are chosen independently of the data $(X_{0,i})_{i = 1,\ldots,n}$.   Then, $(\cev{p}{}^{\hat{\mathfrak{s}}_n}_t)_{t \in [0, \overline{T})}$ are the densities of the backward process driven by the score estimate $\hat{\mathfrak{s}}_n$. In particular, $\cev{p}{}^{\hat{\mathfrak{s}}_n}_{\overline{T}-\underline{T}}$ is the density of the generated data obtained from stopping the backward process early at time $\overline{T} - \underline{T}$. Assessing the quality of the generated samples therefore boils down to analysing the distance between the distribution induced by $p_0$ and the (random) distribution induced by $\cev{p}{}^{\hat{\mathfrak{s}}_n}_{\overline{T}-\underline{T}}$.

In this paper, we use the total variation distance as the divergence measure. For two probability measures $\mathbf{P},\mathbf{Q}$ on $\overline{D}$ with Lebesgue densities $p,q$, the total variation distance is denoted by
\[\TV(p,q) \equiv \TV(\mathbf{P},\mathbf{Q}) \coloneqq \sup_{A \in \mathcal{B}(\overline{D})} \lvert \mathbf{P}(A) - \mathbf{Q}(A)\rvert.\]
Our main result is the following.

\begin{theorem} \label{theo:main}
Assume that $p_0 = \tilde{p}_0 + \alpha$ for some nonnegative $\tilde{p}_0 \in H^s_c(D)$ and $\alpha > 0$, where $s \in \N \cap (d/2,\infty)$. Let
\[\underline{T} \asymp  n^{-\frac{2s}{\beta(2s+d)}} \quad \text{and} \quad \overline{T} = \frac{s}{\lambda_1(2s+d)} \log n,\] 
where $\beta = 1$ if $s > d/2+1$, $\beta = s-d/2$ if $s \in (d/2,d/2+1)$ and for $s= d/2+1$, $\beta$ can be arbitrarily chosen in $(0,1)$. 
Then, there exists a class of  neural networks  
\[\mathcal{S} = \big\{\mathfrak{s} \in \Phi\big(L(n),W(n),S(n),B(n)\big): \lVert \mathfrak{s}(\cdot,t)\rVert_D \leq C(t^{-1/2} \vee 1) \, \forall t \in [\underline{T}, \overline{T}] \big\},\] 
for some global constant $C> 0$, with network sizes
\begin{align*}
L(n)
&\lesssim\log n\log\log n, \\
\n{W(n)}_{\infty}
&\lesssim  n^\frac{d}{2s+d}(\log n)^2, \\
S(n)
&\lesssim n^\frac{d}{2s+d}(\log n)^3,\quad\text{and} \\
B(n)
&\lesssim {\color{black}n^{\frac{2s}{\beta(2s+d)}}},
\end{align*}
such that, for $\hat{\mathfrak{s}}_n$ given by \eqref{eq:score_est}, it holds for $n$ large enough that
\begin{equation}\label{eq:conv_rate}
\E\big[\TV(p_0,\cev{p}{}^{\hat{\mathfrak{s}}_n}_{\overline{T}- \underline{T}})\big] \lesssim n^{-\frac{s}{2s+d}} (\log n)^{3} (\log \log n)^{1/2}.
\end{equation}
\end{theorem}
The proof of this theorem is prepared in the following sections. There, we will always work under the assumption on $p_0$ from the theorem, that is, we assume 
\begin{enumerate}[label = ($\mathcal{H}$0), ref = ($\mathcal{H}$0)]
\item there exist a nonnegative function $\tilde{p}_0 \in H^s_c(D)$, for $s \in \N \cap (d/2,\infty)$, and $\alpha > 0$ such that $p_0 = \tilde{p}_0 + \alpha$,
\end{enumerate}
without further comment. Note that since $s > d/2$, the Sobolev imbedding theorem yields the continuous imbedding $H^s_c(D) \hookrightarrow \mathcal{C}^{0,\beta}(D)$, where $\beta$ is defined depending on $s$ and $d$ as stated in Theorem \ref{theo:main}. Thus, we may consider $p_0$ as a continuous function such that there exist some Hölder constants $\beta \in (0,1],c_\beta \in (0,\infty)$, fixed throughout the rest of the paper, such that 
\begin{equation}\label{eq:hoelder}
\lvert p_0(x) - p_0(y) \rvert \leq c_\beta \lvert x - y \rvert^\beta, \quad x,y \in D.
\end{equation}

\subsection{Error decomposition}\label{subsec:decomp}
Let $\overline{p}_t^{\mathfrak{s}}$ be the density at time $t$ of the solution to the SDE \eqref{eq:rev_score_approx} with initial condition $\overline{X}^{\mathfrak{s}}_0 \sim p_{\overline{T}}$ (independent of the data $(X_{0,i})_{i=1,\ldots,n}$) replacing $\overline{X}^{\mathfrak{s}}_0 \sim \mathcal{U}(\overline{D})$. In particular, we then notice that $\overline{p}{}^{\mathfrak{s}^\circ}_{t} = p_{\overline{T} - t}$ and therefore $\overline{p}{}^{\mathfrak{s}^\circ}_{\overline{T} - \underline{T}} = p_{ \underline{T}}$.  As for unconstrained diffusion models, cf.\ \cite{oko23}, the triangle inequality for the total variation distance then implies the error decomposition bound
\begin{equation}\label{eq:error_decomp}
\begin{split}
\E\big[\operatorname{TV}(p_0,\cev{p}{}^{\hat{\mathfrak{s}}_n}_{\overline{T}- \underline{T}})\big]&\le \TV(p_0,p_{\underline{T}}) + \TV(\PP(X_{\overline{T}} \in \cdot \mid X_0 \sim p_0), \mathcal{U}(\overline{D}))\\
&\quad+ \E\big[\TV(\overline{p}{}^{\mathfrak{s}^\circ}_{\overline{T} - \underline{T}}, \overline{p}{}^{\hat{\mathfrak{s}}_n}_{\overline{T} - \underline{T}}) \big].
\end{split}
\end{equation}
The generalisation error therefore splits into three separate error contributions: 
\begin{enumerate} 
\item the first term represents the error induced by stopping early the backward process initialised by the true forward terminal density $p_{\overline{T}}$  at time $\overline{T} - \underline{T}$; 
\item the second term is the error associated to starting the backward process in its stationary distribution instead of $p_{\overline{T}}$;
\item the third term quantifies the error coming from running the backward process with the drift determined by the estimated score $\hat{\mathfrak{s}}_n$ instead of the true score $\mathfrak{s}^\circ$. 
\end{enumerate}
We start with controlling the first two terms before treating the most challenging score approximation error. The early stopping contribution to the error decomposition is controlled via small time heat kernel bounds for the transition densities in the following lemma.

\begin{lemma}\label{lem:early_stop}
There exists a constant $C$ depending only on $f, d, D, \beta$ and $c_\beta$ such that
	\[
		\mathrm{TV}(p_0, p_{\underline{T}})
		\le C\underline{T}^{\beta/2},\qquad \underline{T}\le 1.
	\]
\end{lemma}
\begin{proof}
Fix $0<t\le 1$.
We need to show that $\tfrac{1}{2}\n{p_t-p_0}_{L^1}\le Ct^{\beta/2}$.
To do so, note that the reversibility of $X$ implies the symmetry of its transition densities. Hence,
\[
p_t(x)=\int p_0(y)q_t(y, x)\diff{y}
		=\int p_0(y)q_t(x, y)\diff{y}
		=\mathbb{E}[p_0(X_t)\mid X_0=x],\qquad x\in D.
  \]
Using the Hölder continuity stated in \eqref{eq:hoelder}, we therefore obtain
\[
		\lvert p_t(x)-p_0(x)\rvert
		=\lvert\mathbb{E}[p_0(X_t)-p_0(X_0)\mid X_0=x]\rvert
		\le c_\beta \mathbb{E}[\lvert X_t-X_0\rvert^\beta\mid X_0=x],\qquad x\in D.
\]
Furthermore, by \citet[Corollary 3.2.9]{davies89}, we have the small time Gaussian heat kernel bound
\[
q_t(x, y)\le C_0\frac{1}{t^{d/2}}\exp\Big(-C_1\frac{|x-y|^2}{t}\Big),
\quad x,y\in D,
\]
where $C_0, C_1\ge0$ are constants depending only on $f, d$ and $D$.
Thus, for $x\in D$,
\[
		\mathbb{E}[\lvert X_t-X_0\rvert^\beta\mid X_0=x]
		=\int_Dq_t(x, y)\lvert x-y\rvert^\beta\diff{y}
		\le C_0\frac{1}{t^{d/2}}\int_D\lvert x-y\rvert^\beta \exp\Big(-C_1\frac{\lvert x-y \rvert^2}{t}\Big)\diff{y}.
\]
Next, note that
	\begin{align*}
		\int_D\lvert x-y\rvert^\beta \exp\Big(-C_1\frac{\lvert x-y\rvert^2}{t}\Big)\diff{y}
		&\le\int_{\R^d}\lvert y\rvert^\beta\exp\Big(-C_1\frac{\lvert y\rvert^2}{t}\Big)\diff{y} \\
		&=\kappa_d\int_{0}^{\infty}r^{\beta + d -1}\exp\Big(-C_1\frac{r^2}{t}\Big)\diff{r} \\
		&=\frac{\kappa_d}{2}t^{\frac{\beta + d}{2}}C_1^{-\frac{\beta+d}{2}}\int_{0}^{\infty}u^{\frac{\beta + d }{2}-1}\mathrm{e}^{-u}\diff{u} \\
		&=C_2t^{\frac{\beta + d}{2}},
	\end{align*}
	where $\kappa_d$ is the surface area of $S^{d-1}$ and $C_2=\frac{\kappa_d}{2}C_1^{-\frac{\beta+d}{2}}\Gamma(\frac{\beta + d}{2})$.
	Finally, since $D$ is bounded, the result follows by setting $C=c_\beta C_0C_2 \mathrm{Leb}(D)/2$.
\end{proof}

The error contribution from starting the backward process uniformly on $D$ is controlled in terms of the spectral gap $\lambda_1$ of $\mathcal{A}$ defined in \eqref{eq:generator}, which can be lower bounded by $\lambda_1 \geq f_{\min}/C_{\text{P}}(D)$, where $C_{\text{P}}(D)$ is the Poincaré constant of the domain $D$.

\begin{lemma} \label{lem:ergodic_err}
It holds that
\[\TV(\PP(X_{\overline{T}} \in \cdot \mid X_0 \sim p_0), \mathcal{U}(\overline{D})) \leq \frac{\sqrt{\operatorname{Leb}(D)}}{2} \lVert p_0 \rVert_{L^2} \mathrm{e}^{-\lambda_1 \overline{T}}, \quad \overline{T} > 0.\]
\end{lemma}
\begin{proof} 
Let $t > 0$. 
The Cauchy--Schwarz inequality implies that
\begin{align*} 
\TV(\PP(X_t \in \cdot \mid X_0 \sim p_0), \mathcal{U}(\overline{D})) &= \frac{1}{2} \int_D \big\lvert p_t(x) - \tfrac{1}{\operatorname{Leb}(D)} \big\rvert \diff{x}\\
&\leq \frac{\sqrt{\operatorname{Leb}(D)}}{2} \Big( \int_D \big\lvert p_t(x) - \tfrac{1}{\operatorname{Leb}(D)} \big\rvert^2 \diff{x} \Big)^{1/2}.
\end{align*}
By the spectral decomposition of $q_t(y,x)$, it follows for $a_k \coloneqq \langle p_0,e_k\rangle_{L^2}$ that
\begin{align*} 
\int_D \big\lvert p_t(x) - \tfrac{1}{\operatorname{Leb}(D)}\big\rvert^2 \diff{x} &= \int_D \Big\lvert \sum_{k \geq 1} \mathrm{e}^{-\lambda_k t} e_k(x) a_k \Big\rvert^2 \diff{x} \\ 
&= \sum_{k\geq 1} a_k^2 \mathrm{e}^{-2\lambda_k t}\\ 
&\leq \lVert p_0 \rVert_{L^2}^2 \mathrm{e}^{-2\lambda_1 t},
\end{align*}
where we used $\lambda_0 = 0$, $e_0 \equiv \operatorname{Leb}(D)^{-1/2}$ and $\operatorname{Leb}(D)^{1/2}a_0 = \langle p_0, 1\rangle_{L^2}  = 1$ for the first line and orthonormality of $(e_k)_{k \geq 0}$ for the second one. This yields the claim.
\end{proof}
We now move on to the treatment of score approximation error in the next section.

\subsection{Score matching error}\label{subsec:score_match}
In this section, $\mathcal{S}$ denotes a generic neural network class that shall be exactly calibrated at the end of the section for our score approximation purposes. Recall that the score estimator $\hat{\mathfrak{s}} = \hat{\mathfrak{s}}_n$ is defined according to \eqref{eq:score_est}. 

By Girsanov's theorem, cf.\ Theorem \ref{theo:girsa}, and Pinsker's inequality, the third term in the decomposition \eqref{eq:error_decomp} is controlled by 
{\color{black}
\begin{equation}\label{eq:gen_err}
\begin{split}
&\Big(\frac{1}{2}\E\Big[\int_{\underline{T}}^{\overline{T}} \int_{D} f(x) \lvert \hat{\mathfrak{s}}(x,t) - \nabla \log p_t(x) \rvert^2 p_t(x) \diff{x} \diff{t}  \Big] \Big)^{1/2}\\
&\quad\asymp \Big(\E\Big[\int_{\underline{T}}^{\overline{T}} \int_{D} \lvert \hat{\mathfrak{s}}(x,t) - \nabla \log p_t(x) \rvert^2 p_t(x) \diff{x} \diff{t}  \Big]\Big)^{1/2},
\end{split}
\end{equation}
where we used that $0< f_{\min} \leq f(x) \leq \lVert f \rVert_{\overline{D}} < \infty$ for all $x \in D$.
}
The key to bounding this term is the equivalence between explicit and denoising score matching, i.e.,
\begin{equation}\label{eq:expl_denoise}
\begin{split}
&\int_{\underline{T}}^{\overline{T}}\int_D \lvert \sco(y,t) - \nabla \log p_t(y) \rvert^2 p_t(y) \diff{y} \diff{t} \\
&\quad = \int_{\underline{T}}^{\overline{T}}\int_{D^2} \lvert \sco(y,t) - \nabla_y \log q_t(x,y) \rvert^2 q_t(x,y) p_0(x) \diff{x} \diff{y} \diff{t} + C\\ 
&\quad=\E[L_{\sco}(X_0)] + C, 
\end{split}
\end{equation}
where 
\[C = \int_{\underline{T}}^{\overline{T}}\int_D \lvert \nabla \log p_t(y) \rvert^2 p_t(y) \diff{y} \diff{t} - \int_{\underline{T}}^{\overline{T}}\int_{D^2} \lvert \nabla_y \log q_t(x,y) \rvert^2 q_t(x,y) p_0(x) \diff{x} \diff{y} \diff{t} \leq 0\] 
is a constant that is independent of $\sco$. Note that \eqref{eq:expl_denoise} is valid in our reflected diffusion model by the same arguments as in \cite{vincent11}, see also the proof of Lemma C.3 in \cite{oko23}.

Using \eqref{eq:expl_denoise}, the generalisation loss \eqref{eq:gen_err} can be bounded in terms of the minimal score approximation error over the class $\mathcal{S}$ and the complexity of the induced function class $\mathcal{L} \coloneqq \{L_{\sco} : \sco \in \mathcal{S}\}$ for a desired precision level $\delta$.

\begin{theorem}\label{theo:score_loss_emp}
Suppose that $\sup_{\sco \in \mathcal{S} \cup \{\sco^\circ\}} \lVert L_{\sco} \rVert_{D} \leq C(\mathcal{L}) < \infty$.  Then, for any $\delta > 0$ such that $\mathcal{N}(\mathcal{L},\lVert \cdot \rVert_D, \delta) \geq 3$, it holds that
\begin{equation*} 
\begin{split}
&\E\Big[\int_{\underline{T}}^{\overline{T}} \int_{\overline{D}} \lvert \hat{\mathfrak{s}}(x,t) - \nabla \log p_t(x) \rvert^2 p_t(x) \diff{x} \diff{t}  \Big] \\ 
&\quad \leq 2 \inf_{\sco \in \mathcal{S}} \int_{\underline{T}}^{\overline{T}} \int_{D} \lvert \sco(x,t) - \nabla \log p_t(x) \rvert^2 p_t(x) \diff{x} \diff{t} + 2\frac{C(\mathcal{L})}{n}\Big(\frac{145}{9}\log \mathcal{N}(\mathcal{L}, \lVert \cdot \rVert_{D}, \delta) + 160\Big)\\
&\qquad+ 5\delta.
\end{split}
\end{equation*}
\end{theorem}
The proof of this theorem is in principle the same as that of Theorem C.4 in \cite{oko23}, but let us emphasize that larger numeric constants appear in our statement. Apart from a few simple typos in \cite{oko23}, the main reason for this is a small gap in the proof of \citet[Theorem C.4]{oko23}, which has recently been pointed out in \cite{yakovlev25}, and which requires fixing. More precisely, \cite{oko23}  claim that the excess loss satisfies
\[\E\big[(L_\sco(X_0) - L_{\sco^\circ}(X_0))^2\big] \leq C(\mathcal{L})\E[L_\sco(X_0) - L_{\sco^\circ}(X_0)],\]
which is unjustified because $L_\sco(x) - L_{\sco^\circ}(x)$ is not necessarily non-negative pointwise. This is not dramatic however, since we can show that instead
\begin{equation}\label{eq:bernstein_condition}
\E\big[(L_s(X_0) - L_{s^\circ}(X_0))^2\big] \leq 4 C(\mathcal{L})\E[L_s(X_0) - L_{s^\circ}(X_0)],
\end{equation}
holds, i.e., the  bound from \cite{oko23} is true up to a universal multiplicative constant that does not matter in an essential way for the remainder of the proof. 
Using the terminology of \cite{bartlett05}, \eqref{eq:bernstein_condition} shows that the denoising score matching excess loss satisfies a \textit{Bernstein condition}, which for a variety of problems in the empirical risk minimisation literature has been identified as a crucial ingredient  to obtain minimax optimal convergence rates for empirical risk minimisers. For reasons of reproducibility  in other modelling contexts, we prove \eqref{eq:bernstein_condition}  in Appendix \ref{app:bernstein} in a general Markovian framework. 

To deal with the stochastic error in the upper bound, it is essential to control both the uniform loss upper bound $C(\mathcal{L})$ and the covering number $\mathcal{N}(\mathcal{L},\lVert \cdot \rVert_D,\delta)$. The size of the networks from the generic class $\mathcal{S}$ which is required to have a sufficient bound on the approximation error 
\[\inf_{\sco \in \mathcal{S}} \int_{\underline{T}}^{\overline{T}} \int_{D} \lvert \sco(x,t) - \nabla \log p_t(x) \rvert^2 p_t(x) \diff{x} \diff{t}\]
translates directly into covering number bounds on $\mathcal{N}(\mathcal{L}, \lVert \cdot \rVert_{D}, \delta)$, as the following result shows. 

\begin{lemma}\label{lem:cover}
If for any $t > 0$, $\sup_{\sco \in \mathcal{S}} \lVert  \sco(\cdot,t)  \rVert_{D} \leq C(\mathcal{S})(t^{-1/2} \vee 1)$, for some finite constant $C(\mathcal{S})$, then there exists a constant $c$, depending only on $d,f$ and $D$, such that, for any $\delta > 0$,
\[\mathcal{N}(\mathcal{L}, \lVert \cdot \rVert_D, \delta) \leq \mathcal{N}\big(\mathcal{S}, \lVert  \cdot  \rVert_{D \times [\underline{T},\overline{T}]}, \tfrac{\delta}{c C(\mathcal{S})\overline{T} }\big).\]
\end{lemma}
\begin{proof} 
For $\tilde \delta>0$, let $\sco_1,\ldots,\sco_N\colon \R^d\to\R$ be a $\tilde \delta$-net for $\mathcal{S}$ w.r.t.\ $\lVert  \cdot  \rVert_{D \times [\underline{T},\overline{T}]}$.
Denote $h_{\sco}(x;y,t)\coloneqq \sco(y,t)- \nabla_y \log q_t(x,y)$ for $\sco \in \mathcal{S}$.   Let $\sco \in \mathcal{S}$, and choose $\sco^\prime \in \{\sco_1,\ldots,\sco_N\}$ such that $\lVert  \sco - \sco^\prime  \rVert_{D \times [\underline{T},\overline{T}]} \leq \tilde{\delta}$.
Using definition \eqref{eq:lossfct} and the inequality  $\left\lvert\lvert h_\sco \rvert - \lvert h_{\sco'} \rvert \right\rvert(x;y,t) \leq \lvert \sco-\sco'\rvert(y,t) $, we obtain
\begin{equation}\label{eq:bound_cover}
\begin{split}
		\lvert L_{\sco} - L_{\sco^\prime} \rvert(x)
		&\leq \int_{ \underline{T}}^{\overline{T}}\int_{\overline D}\left\lvert\lvert h_\sco \rvert^2-\lvert h_{\sco'}\rvert^2\right\rvert(x;y,t)q_t(x,y)\diff{y}\diff{t}\\
		&= \int_{ \underline{T}}^{\overline{T}}\int_{\overline D}\left\lvert\lvert h_\sco \rvert -\lvert h_{\sco'} \rvert\right\rvert(x;y,t)\left(\lvert h_\sco \rvert + \lvert h_{\sco'} \rvert\right)(x;y,t)q_t(x,y)\diff{y}\diff{t}\\
    	&\lesssim \lVert  \mathfrak{s} - \mathfrak{s}^\prime  \rVert_{D \times [\underline{T},\overline{T}]} \int_{ \underline{T}}^{\overline{T}}\int_{\overline D}\left(\lvert h_\sco \rvert + \lvert h_{\sco'}\rvert\right)(x;y,t)q_t(x,y)\diff{y}\diff{t}\\
    	&\lesssim \tilde{\delta} \Big(\int_{\underline{T}}^{\overline{T}} \sup_{\mathfrak{s} \in \mathcal{S}}\sup_{z \in D} \lvert \mathfrak{s}(z,t) \rvert \diff{t} + \int_{\underline{T}}^{\overline{T}} \int_D \lvert \nabla_y \log q_t(x,y) \rvert q_t(x,y) \diff{y} \diff{t} \Big) \\
&\lesssim \tilde{\delta}\Big(C(\mathcal{S})\overline{T} + \int_{\underline{T}}^{\overline{T}} \int_D \lvert \nabla_y \log q_t(x,y) \rvert q_t(x,y) \diff{y} \diff{t} \Big).
\end{split}
\end{equation}
Using \citet[Theorem 6.19]{ouh05} 
and symmetry of $q_t$, we obtain for some constants $C,\gamma > 0$ only depending on $d,f$ and $D$,
\begin{equation}\label{eq:log_grad}
\begin{split}
\int_D  \lvert \nabla_y \log q_t(x,y) \rvert q_t(x,y) \diff{y} &= \int_D \lvert \nabla_y q_t(x,y) \rvert \diff{y} \\ 
&= \int_D \lvert \nabla_y q_t(y,x) \rvert \diff{y} \\
&\leq Ct^{-1/2} \mathrm{e}^{\gamma t}.
\end{split}
\end{equation}
This shows that
\begin{equation} \label{eq:bound_cover1}
\int_{\underline{T}}^{1} \int_D  \lvert \nabla_y \log q_t(x,y) \rvert q_t(x,y) \diff{y} \diff{t} \lesssim 1.
\end{equation}
Furthermore, as in the proof of \citet[Proposition 3]{nickl23}, we have 
\begin{equation}\label{eq:grad_bound}
\sup_{(x,y) \in D^2} \lvert \nabla_y q_t(x,y) \rvert \lesssim \sum_{j \geq 1} j^{\tau +1/d}\mathrm{e}^{-ctj^{2/d}}, \quad t > 0,
\end{equation}
for some constants $c > 0, \tau > 1/2$, showing also that
\begin{equation}\label{eq:bound_cover2}
\int_1^{\overline{T}}  \int_D  \lvert \nabla \log q_t(x,y) \rvert q_t(x,y) \diff{y} \diff{t} \lesssim 1.
\end{equation}
Plugging \eqref{eq:bound_cover1} and \eqref{eq:bound_cover2} into \eqref{eq:bound_cover}, it follows that $\{L_{\mathfrak{s}_i}: i \in [N]\}$ is an $\tilde{\delta}cC(\mathcal{S})\overline{T}$-covering of $\mathcal{L}$ w.r.t.\ $\lVert \cdot  \rVert_{D}$, where $c$ is some constant depending only on $f,d$ and $D$. This implies the claimed result.
\end{proof}

For the uniform loss upper bound $C(\mathcal{L})$ we again need to deal with the challenge of not having access to a simple analytic expression for the transition densities $q_t(x,y)$. While upper and lower heat kernel bounds are available for $q_t$ and its gradient under specific assumptions on the domain, these are not sufficient to yield appropriate pointwise estimates for the log-gradient $\nabla_y q_t(x,y)$. Instead, we exploit that for our purposes it suffices to have (time)-integrated bounds. The basic idea is best illustrated for the particular case of a constant diffusivity $f \equiv 1$. Then, the generator of the forward process is just the Neumann Laplacian $\Delta$ on $\overline{D}$ which implies that for any $t > 0$ and $x,y \in D$
\[\Delta_y q_t(x,y) = \partial_t q_t(x,y),\]
because  $q_t(x,y)$ is a symmetric fundamental solution to the Neumann heat equation. Furthermore,  
\[\Delta_y \log q_t(x,y) = \frac{\Delta_y q_t(x,y)}{q_t(x,y)} - \lvert \nabla_y \log q_t(x,y) \rvert^2,\]
which together with the above establishes the fundamental relation 
\[\lvert \nabla_y \log q_t(x,y) \rvert^2 - \partial_t \log q_t(x,y) = - \Delta_y \log q_t(x,y),\]
between spatial and temporal log-gradients. Based on this observation, the famous Li--Yau estimate \citep{liyau} establishes that for $f \equiv 1$ it holds that 
\[\lvert \nabla_y \log q_t(x,y) \rvert^2 - \partial_t \log q_t(x,y) \leq \frac{d}{2t}.\]
This result could be directly used for $f \equiv 1$ to establish the bound in the following lemma, but if $f$ is not constant the situation becomes a bit more tricky. While we may always interpret $\Delta_f = \nabla \cdot f\nabla$ as a weighted Neumann Laplacian on the manifold $\overline{D}$ equipped with a Riemannian metric induced by the Riemannian tensor associated to $f$, corresponding Li--Yau type estimates from the literature \citep{qian95,bakry99} require the validation of certain curvature conditions on $\nabla \cdot f \nabla$, which may be hard to check for specific choices of $f$ and $D$. To circumvent this problem, we follow a more elementary approach, which is however still based on the type of reasoning outlined above.

\begin{lemma}\label{lem:scoreloss_bound}
 Assume that $\underline{T} \leq 1$ and $\overline{T} \geq 1$. If for any $t > 0$, $\sup_{\sco \in \mathcal{S}} \lVert \sco(\cdot,t) \rVert_{D} \leq C(\mathcal{S})(t^{-1/2} \vee 1)$, then 
\[\sup_{\sco \in \mathcal{S} \cup \{\sco^\circ\}} \lVert L_{\sco}\rVert_D \lesssim (C(\mathcal{S})^2 \vee 1) (\lvert\log \underline{T}\rvert + \overline{T}).\]
\end{lemma}
\begin{proof} 
Let $\sco \in \mathcal{S} \cup \{\sco^\circ\}$ and $x \in D$ be arbitrarily chosen. By the assumption on the uniform temporal growth of $\sco$, and using that $\int_D q_t(x,y) \diff{y} = 1$
\begin{align*}
L_{\sco}(x) &\leq 2\Big(\int_{\underline{T}}^{\overline{T}} \int_{D} \big(\lvert \sco(y,t)\rvert^2 + \lvert \nabla_y \log q_t(x,y) \rvert^2\big) q_t(x,y) \diff{y}\Big)\\ 
&\lesssim C(\mathcal{S})^2\int_{\underline{T}}^{\overline{T}}  t^{-1} \vee 1 \diff{t} + \int_{\underline{T}}^{\overline{T}} \int_{D}  \lvert \nabla_y \log q_t(x,y) \rvert^2 q_t(x,y) \diff{y}\\ 
&\leq C(\mathcal{S})^2(\lvert \log \underline{T} \rvert + \overline{T}) + \int_{\underline{T}}^{\overline{T}} \int_{D}  \lvert \nabla_y \log q_t(x,y) \rvert^2 q_t(x,y) \diff{y}. 
\end{align*}
To bound the remaining integral involving the log-gradient of the transition density, we first observe that  $\nabla_y \cdot f \nabla_y \log q_t(x,y)$ satisfies
\begin{align*} 
\nabla_y \cdot f \nabla_y \log q_t(x,y) &= \nabla_y \cdot f(y)\frac{\nabla_y q_t(x,y)}{q_t(x,y)} \\
&= \frac{\nabla_y \cdot f \nabla_y q_t(x,y)}{q_t(x,y)} - f(y) \frac{\lvert \nabla_y q_t(x,y) \rvert^2}{q_t(x,y)^2}. 
\end{align*}
Since $q_t(x,y) = q_t(y,x)$ is a fundamental solution to the elliptic PDE $(\nabla \cdot f \nabla -\partial_t) u(y,t) = 0$ on $D$ with Neumann boundary conditions, we can write 
\begin{align*} 
\int_{D} \int_{\underline{T}}^{\overline{T}} \frac{\nabla_y \cdot f\nabla_y q_t(x,y)}{q_t(x,y)} q_t(x,y) \diff{t} \diff{y} &= -\int_{D} \int_{\underline{T}}^{\overline{T}} \underbrace{\frac{\partial_t q_t(x,y)}{q_t(x,y)}}_{=\partial_t \log q_t(x,y)} q_t(x,y) \diff{y}\diff{t}\\
&=  -\int_D\int_{\underline{T}}^{\overline{T}} \partial_t q_t(x,y) \diff{t} \diff{y} = 0.
\end{align*}
Combining these two observations with $0 < f_{\min} \leq f(y) \leq \lVert f \rVert_{\overline{D}} < \infty$ for all $y \in \overline{D}$, and denoting  the generator by $\Delta_f = \nabla \cdot f \nabla $, we obtain
\begin{align*} 
\int_{\underline{T}}^{\overline{T}} \int_{D} \frac{\lvert \nabla_y q_t(x,y) \rvert^2}{q_t(x,y)^2} q_t(x,y) \diff{y} \diff{t} &\asymp -\int_{\underline{T}}^{\overline{T}}  \int_{D} (\nabla_y \cdot f(y)\nabla_y \log q_t(x,y)) q_t(x,y) \diff{y} \diff{t}\\
&= -\int_{\underline{T}}^{\overline{T}} \int_{D} ((\Delta_f + \partial_t) \log q_t(x,y)) q_t(x,y)\diff{y} \diff{t}\\
&= -\E^x\Big[\int_{\underline{T}}^{\overline{T}} (\Delta_f + \partial_t) \log q_t(x, X_t)   \diff{t}\Big] \\
&= \E^x[\log q_{\underline{T}}(x,X_{\underline{T}})] - \E^x[\log q_{\overline{T}}(x,X_{\overline{T}})].
\end{align*}
Noting that $q_t(x,\cdot)$ satisfies the Neumann boundary condition at $\partial{D}$, the last equality is a consequence of Itô's formula, by which 
\begin{align*}
&\log q_{\overline{T}}(x,X_{\overline{T}}) - \log q_{\underline{T}}(x,X_{\underline{T}})\\
&\,= \int_{\underline{T}}^{\overline{T}} (\partial_t + \Delta_f) \log q_t(x,X_t) \diff{t} + \int_{\underline{T}}^{\overline{T}} \sqrt{2f(X_t)} \langle \nabla \log q_t(x,X_t), \diff{W_t} \rangle\\
&\,\quad + \int_{\underline{T}}^{\overline{T}} \underbrace{\langle \nabla_y \log q_t(x,X_t), \nu(X_t) \rangle}_{=0} \diff{\ell_t},
\end{align*}
where the expectation of the stochastic integral is zero because $(y,t) \mapsto \lvert \sqrt{f(y)}\nabla_y \log q_t(x,y)\rvert$ is bounded on $\overline{D}\times [\underline{T}, \overline{T}]$, which follows from the lower bound 
\begin{equation} \label{eq:lower_uni}
\inf_{t \geq t_0, x,y \in \overline{D}} q_t(x,y) > 0, 
\end{equation}
for any $t_0 > 0$, see, e.g., \citet[p.166]{ito92},
and the space-time smoothness of $(t,x,y) \mapsto q_t(x,y)$ on the compact set $\overline{D}\times [\underline{T}, \overline{T}]$. We finish the proof by noting that for some constant $C$, the upper heat kernel bound 
\[q_t(x,y) \lesssim (t^{-d/2} \vee 1) \exp\Big(-C\frac{\lvert x-y \rvert^2}{t}\Big), \quad t > 0, x,y \in \overline{D},\]
from \citet[Theorem 3.29]{davies89} yields 
\[\sup_{x,y \in \overline{D}} \log q_{\underline{T}}(x,y) \lesssim \frac{d}{2} \log \underline{T}^{-1},\]
and, moreover, \eqref{eq:lower_uni} implies $\inf_{x,y \in \overline{D}} q_{\overline{T}}(x,y) \geq c$, for some constant $c$ not depending on $\underline{T},\overline{T}$ since $\overline{T} \geq 1$.
Putting these bounds together, we conclude that 
\[\E^x[\log q_{\underline{T}}(x,X_{\underline{T}})] - \E^x[\log q_{\overline{T}}(x,X_{\overline{T}})] \lesssim 1 + \log \underline{T}^{-1},\]
and therefore also 
\[\int_{\underline{T}}^{\overline{T}} \int_{\mathcal{X}} \frac{\lvert \nabla_y q_t(x,y) \rvert^2}{q_t(x,y)^2} q_t(x,y) \diff{y} \diff{t} \lesssim 1 + \log \underline{T}^{-1}.\]
\end{proof}

\subsection{Bounding the approximation error}\label{subsec:approx}
Before going into details, we first outline here our general strategy for bounding the approximation error:
\begin{enumerate}
    \item approximate the true score by truncation of the spectral decomposition $h_N$, i.e., approximate $\nabla_x\log p_t(x)$ by $\nabla_x\log h_N(x, t)$;
    \item approximate $h_N$ and $\nabla_xh_N$ by neural networks on $[\underline{T}, \overline{T}]$ via
    \begin{enumerate}[label=(\alph*)]
        \item dividing $[\underline{T},\overline{T}]$ into sub-intervals of increasing length, totalling a number of intervals on the order of $\log N$;
        \item on each sub-interval, fix a number of time-points $\{t_i\}$, also on the order of $\log N$, and at each of these, make an approximation of $h_N(t_i)$ and $\nabla_xh_N(t_i)$ using existing results;
        \item extend these discrete approximations to the entire sub-interval using polynomial interpolation;
        \item combine approximations on each sub-interval into one final approximation via a partition of unity;
    \end{enumerate}
    \item combine the first two steps, along with general results on neural networks, to achieve a neural network approximation of $\nabla_x\log p_t$.
\end{enumerate}
{\color{black}
At this point, we comment on how and why our approximation strategy differs from those in, for example, \cite{oko23}.
The authors there assume a Gaussian transition kernel with a density that is comparatively easy to approximate using neural networks. The difficulty then lies in approximating the initial density $p_0$ and the convolution of the two.
By contrast, in our spectral composition, the influence of $p_0$ enters the function only as the weights $\langle p_0, e_j\rangle_{L^2}$. The growth of these weights encodes the smoothness of $p_0$ and thereby influences the truncation parameter $N$ but they need no approximation by neural networks. Thus, the difficulty lies instead in approximating the eigenfunctions $e_j$ themselves.
One might think that one could use existing results, e.g.\ from \cite{suzuki19,schmidthieber20}, to approximate each $e_j$ and the individual time components $t\mapsto \e^{-t\lambda_j}$ separately, and then sum them together.
However, using these existing results, the sparsity constraint of each summand would be of order at least $N$, and thus the sparsity constraint of the network as a whole would be at least $N^2$.
This is problematic, since, according to \citet[Lemma C.2]{oko23}, the sparsity constraint enters exponentially in the covering number of the associated class and, consequently, by Theorem \ref{theo:score_loss_emp}, linearly in the generalisation error.
This is why we employ a different strategy: we approximate the entire sum at fixed time points and use polynomials to interpolate between them over time. This ultimately gives us a sparsity constraint of order $N\,\mathrm{Poly}(\log N)$.}
Following this general strategy, we can prove the following score approximation result. 


\begin{theorem}
\label{theorem:approx_nabla_log_h_N}
Let $0<\underline{T}<\overline{T}$ and $n\in\N$ sufficiently large be given with $\underline{T}\in\mathrm{Poly}(n^{-1})$. 
Then, there exists a neural network $\varphi_{\sco}\in\Phi(L(n), W(n), S(n), B(n))$ satisfying
\begin{equation}
\label{eq:err_rate_score_approximation}
\int_{\underline{T}}^{\overline{T}}\int_D
\big|\varphi_{\sco}(x,t)-\nabla_x\log p_t(x)\big|^2p_t(x)\,\mathrm{d}x\,\mathrm{d}t
\lesssim n^{-\frac{2s}{2s+d}}(\log n)^2(\overline{T}+\log(\underline{T}^{-1})).
\end{equation}
The size of the network is evaluated as 
\begin{align*}
L(n)
&\lesssim\log n\log\log n, \\
\n{W(n)}_{\infty}
&\lesssim M n^\frac{d}{2s+d}\log n, \\
S(n)
&\lesssim M n^\frac{d}{2s+d}(\log n)^2,\quad\text{and} \\
B(n)
&\lesssim {\color{black}\frac{\sqrt{n}}{\log n}}\vee\frac{1}{\underline{T}},
\end{align*}
where $M\in O(\lvert\log\tfrac{\overline{T}}{\underline{T}}\rvert)$. 
Furthermore, the network can be chosen such that there exists a constant $C<\infty$ depending only on $p_0$ and $D$ such that $|\varphi_\sco(x, t)|\le \frac{C}{\sqrt{t}}$ for all $t\in[\underline{T}, \overline{T}]$ and $x\in D$.
\end{theorem}

As alluded to above, the idea is to break up the score approximation error by using the spectral score representation \eqref{eq:score_spectral} and to reduce this to the problem of approximating  $\nabla \log h_N(x,t) = \nabla h_N(x,t)/h_N(x,t)$, where, for $N \in \N$ and $t \in [0,T]$, the truncated series $h_N(t) = h_N(\cdot, t)$ is given by 
\begin{equation}
    \label{eq:hN_def}
    h_N(t) \coloneqq \sum_{j=0}^N \mathrm{e}^{-t\lambda_j} \langle p_0,e_j \rangle_{L^2} e_j.
\end{equation} 

It will then be crucial to choose the cutoff value $N$ of the right order in terms of $n$ to balance the tradeoff between the error incurred through the truncation procedure and the increased approximation quality of $\nabla \log h_N$ in terms of neural networks for smaller $N$. Our analysis will demonstrate that for the desired approximation accuracy $n^{-s/(2s+d)}$, the choice $N \asymp n^{d/(2s+d)}$ is appropriate.

\paragraph{Step 1: Bounding the truncation loss}
We start by considering the approximation properties of $h_N$ and $\nabla_x h_N$. To this end, let us introduce the homogeneous Sobolev space of order $s$ that is induced by the eigendecomposition of $-\nabla \cdot f\nabla$ via
\[\bar{H}^s(D) \coloneqq \big\{\phi \in L^2_0(D): \lVert \phi \rVert_{\bar{H}^s}^2 \coloneqq \sum_{j \geq 1} \lambda_j^s \langle \phi, e_j \rangle^2_{L^2} < \infty \big\}.\]

\begin{lemma}\label{lem:bound31}
It holds that
\[\int_{\underline{T}}^{\overline{T}} \lVert p_t - h_N(t) \rVert_{H^1}^2 \diff{t} \lesssim \lVert p_0 - \tfrac{1}{\operatorname{Leb}(D)}\rVert_{H^s}^2 N^{-2s/d}.\]
\end{lemma}
\begin{proof} 
Arguing as in \citet[Proposition 3]{nickl23}, we see that $p_t,h_N(t) \in H^k(D)$ for any $k \in \N$. 
Since $p_t-h_N(t) \in L^2_0$, \citet[Proposition 2]{nickl23} shows that $p_t - h_N(t) \in \bar{H}^1(D)$ and $\lVert p_t - h_N(t) \rVert_{\bar{H}^1} \asymp \lVert p_t - h_N(t) \rVert_{H^1}$. Thus, 
\begin{align*} 
\int_{\underline{T}}^{\overline{T}} \lVert p_t - h_N(t) \rVert_{H^1}^2 \diff{t} 
&\asymp \int_{\underline{T}}^{\overline{T}} \lVert p_t - h_N(t) \rVert_{\bar{H}^1}^2 \diff{t} \\ 
&= \sum_{j \geq N+1} \int_{\underline{T}}^{\overline{T}} \lambda_j \mathrm{e}^{-2\lambda_j t} \diff{t} \, \langle p_0, e_j \rangle^2_{L^2} \\
&\leq \sum_{j \geq N+1} \langle p_0, e_j \rangle^2_{L^2}.
\end{align*}
As $p_0 \in H^s_c(D)/\R$, it holds $p_0-\tfrac{1}{\operatorname{Leb}(D)} \in H^s_c(D)\slash \R \cap L^2_0(D)$. 
Furthermore, by \citet[Proposition 2]{nickl23}, we have $H^s_c(D)\slash \R \cap L^2_0 \subset \bar{H}^s(D)$ and $\lVert \phi \rVert_{H^s} \asymp \lVert \phi \rVert_{\bar{H}^s}$ for $\phi \in \bar{H}^s(D)$, implying that $\lVert p_0 - \tfrac{1}{\operatorname{Leb}(D)} \rVert_{H^s} \asymp \lVert p_0 -\tfrac{1}{\operatorname{Leb}(D)}\rVert_{\bar{H}^s}$. 
Using this and $\lambda_j \asymp j^{2/d}$, it follows
\begin{equation}\label{eq:markov_spectral}
\begin{split}
\sum_{j \geq N+1} \langle p_0, e_j \rangle_{L^2}^2 \leq \frac{\sum_{j = 1}^{\infty} \langle p_0,e_j \rangle^2_{L^2} j^{2s/d}}{(N + 1)^{2s/d}} &= \frac{\sum_{j = 1}^{\infty} \langle p_0 - \tfrac{1}{\operatorname{Leb}(D)},e_j \rangle^2_{L^2} j^{2s/d}}{(N + 1)^{2s/d}}\\
&\lesssim \lVert p_0 - \tfrac{1}{\operatorname{Leb}(D)}\rVert_{H^s}^2 N^{-2s/d}.
\end{split}
\end{equation}
\end{proof}

This result allow us to study the approximation quality of (a truncated version of) $\nabla \log h_N$. 
\begin{proposition} \label{prop:approx_spectral0}
It holds that
\[\int_{\underline{T}}^{\overline{T}} \int_{D} \Big\lvert \nabla_x \log p_t(x) - \frac{\nabla_x h_N(x,t)}{h_N(x,t) \vee \alpha} \Big\rvert^2 p_t(x) \diff{x} \diff{t} \lesssim C(p_0)(1 + \lambda_1^{-1})\alpha^{-4} \log(\underline{T}^{-1}) N^{-2s/d},\]
where 
\[C(p_0) = \lVert p_0 \rVert_\infty\lVert p_0 - \tfrac{1}{\operatorname{Leb}(D)} \rVert^2_{H^s} (1+\lVert p_0 \rVert^2_{H^s}).\]
\end{proposition}
\begin{proof} 
Symmetry of the transition densities implies that for any $t \geq 0, x \in D$ it holds that
\[p_t(x) = \int_D p_0(y) q_t(y,x) \diff{y} = \int_D p_0(y) q_t(x,y) \diff{y} \leq \lVert p_0 \rVert_\infty \int_D q_t(x,y) \diff{y} = \lVert p_0 \rVert_\infty,\]
and, similarly,
\[p_t(x) \geq \alpha.\]
It follows that
\begin{equation}\label{eq:decomp2}
\begin{split}
&\int_{\underline{T}}^{\overline{T}} \int_{D} \Big\lvert \nabla_x \log p_t(x) - \frac{\nabla_x h_N(x,t)}{h_N(x,t)\vee \alpha} \Big\rvert^2 p_t(x) \diff{x} \diff{t}\\
&\quad\leq \lVert p_0 \rVert_\infty \int_{\underline{T}}^{\overline{T}} \Big\lVert \Big\lvert \nabla_x \log p_t - \frac{\nabla_x h_N(t)}{h_N(t) \vee \alpha} \Big\rvert \Big\rVert^2_{L^2}\diff{t}\\ 
&\quad\leq 2\lVert p_0 \rVert_\infty \Big(\int_{\underline{T}}^{\overline{T}} \Big\lVert \Big\lvert \frac{\nabla_x (p_t - h_N(t))}{h_N(t) \vee \alpha} \Big\rvert \Big\rVert^2_{L^2}\diff{t}+ \int_{\underline{T}}^{\overline{T}} \Big\lVert \Big\lvert \nabla_x p_t\frac{ p_t - (h_N(t) \vee \alpha)}{p_t(h_N(t) \vee \alpha)} \Big\rvert \Big\rVert^2_{L^2}\diff{t} \Big) \\ 
&\quad\leq 2\lVert p_0 \rVert_\infty \Big(\frac{1}{\alpha^2}\int_{\underline{T}}^{\overline{T}} \big\lVert \big\lvert \nabla_x (p_t - h_N(t)) \big\rvert \big\rVert^2_{L^2}\diff{t}+ \frac{1}{\alpha^4}\int_{\underline{T}}^{\overline{T}} \big\lVert \big\lvert \nabla_x p_t (p_t - (h_N(t) \vee \alpha)) \big\rvert \big\rVert^2_{L^2}\diff{t} \Big).
\end{split}
\end{equation}
By Lemma \ref{lem:bound31}, we obtain for the first term
\begin{equation}\label{eq:decomp21}
\begin{split}
\frac{1}{\alpha^2} \int_{\underline{T}}^{\overline{T}} \big\lVert \big\lvert \nabla_x (p_t - h_N(t))\big\rvert \big\rVert^2_{L^2}\diff{t} &\leq  \alpha^{-2} \int_{\underline{T}}^{\overline{T}} \lVert p_t - h_N(t) \rVert_{H^1}^2 \diff{t} \\
&\lesssim \lVert p_0 - \tfrac{1}{\operatorname{Leb}(D)} \rVert^2_{H^s} \alpha^{-2} N^{-2s/d}.
\end{split}
\end{equation}
To bound the second term, note first that 
\begin{align*}
 \lVert p_t - \tfrac{1}{\operatorname{Leb}(D)} \rVert_{\bar{H}^{s+1}}^2 = \sum_{j=1}^\infty \mathrm{e}^{-2\lambda_j t} \lambda_j^{s+1} \langle p_0,e_j \rangle^2_{L^2}&\leq \frac{1}{2t}\sum_{j \geq 1} \lambda_j^s \langle p_0,e_j \rangle^2_{L^2} = \frac{2}{t}\lVert p_0 -\tfrac{1}{\operatorname{Leb}(D)}\rVert_{\bar{H}^s}^2\\ &\asymp \frac{2}{t}\lVert p_0 -\tfrac{1}{\operatorname{Leb}(D)}\rVert_{H^s}^2 < \infty, \quad t > 0.   
\end{align*}
Therefore, for any $i \in [d]$,
\begin{align*}
\lVert \partial_{x_i} p_t \rVert_{H^{s}} = \lVert \partial_{x_i} (p_t - \tfrac{1}{\operatorname{Leb}(D)}) \rVert_{H^{s}} \leq  \lVert p_t - \tfrac{1}{\operatorname{Leb}(D)} \rVert_{H^{s+1}} &\asymp \lVert p_t - \tfrac{1}{\operatorname{Leb}(D)} \rVert_{\bar{H}^{s+1}}\\
&\lesssim \frac{1}{\sqrt{t}}\lVert p_0-\tfrac{1}{\operatorname{Leb}(D)} \rVert_{H^s}. 
\end{align*}
Since $s > d/2$, the Sobolev imbedding theorem yields
\begin{equation}\label{eq:uni_der2}
\sup_{i \in [d]} \lVert \partial_{x_i} p_t \rVert_\infty \lesssim t^{-1/2}\lVert p_0 -\tfrac{1}{\operatorname{Leb}(D)} \rVert_{H^s}.
\end{equation}
Consequently, for $N$ as above, 
\begin{equation}\label{eq:decomp22}
\begin{split}
\int_{\underline{T}}^{\overline{T}} \big\lVert \big\lvert \nabla_x p_t (p_t - (h_N(t) \vee \alpha)) \big\rvert \big\rVert^2_{L^2}\diff{t} &\lesssim \lVert p_0 -\tfrac{1}{\operatorname{Leb}(D)}\rVert_{H^s}^2\int_{ \underline{T}}^{\overline{T}} \frac{1}{t}\lVert p_t - (h_N(t) \vee \alpha) \rVert_{L^2}^2 \diff{t}\\ 
&\leq \lVert p_0 -\tfrac{1}{\operatorname{Leb}(D)}\rVert_{H^s}^2\int_{ \underline{T}}^{\overline{T}} \frac{1}{t}\lVert p_t - h_N(t) \rVert_{L^2}^2 \diff{t}\\ 
&= \lVert p_0 -\tfrac{1}{\operatorname{Leb}(D)}\rVert_{H^s}^2 \int_{ \underline{T}}^{\overline{T}} \frac{1}{t}\mathrm{e}^{-2\lambda_1 t} \diff{t} \sum_{j \geq N+1}\langle p_0, e_j \rangle^2_{L^2}\\
&\lesssim \lVert p_0 - \tfrac{1}{\operatorname{Leb}(D)} \rVert_{H^s}^4 N^{-2s/d}\Big( \int_{\underline{T}}^1 \frac{1}{t} \diff{t} + \int_1^{\overline{T}} \mathrm{e}^{-\lambda_1 t} \diff{t}\Big)\\ 
&\leq \lVert p_0 - \tfrac{1}{\operatorname{Leb}(D)} \rVert_{H^s}^4(1 + 1/\lambda_1) \log(\underline{T}^{-1}) N^{-2s/d},
\end{split}
\end{equation}
where we used \eqref{eq:markov_spectral} and that $\lvert p_t(x) - (h_N(x,t) \vee \alpha) \rvert \leq \lvert p_t(x) - h_N(x,t) \rvert$ since $p_t(x) \geq \alpha$. 
Combining \eqref{eq:decomp2}, \eqref{eq:decomp21} and \eqref{eq:decomp22} yields the assertion.
\end{proof}

\paragraph{Step 2: Approximation of truncated score by neural networks}
Given the previous result, it remains to show that (the truncated version of) $\nabla_x\log h_N$ with $h_N$ as defined in \eqref{eq:hN_def} can be well approximated by a neural network whose size is quantified in terms of $N$. 
To this end, we will make repeated use of the following fundamental observations about compositions and parallelisations of the type of ReLU neural networks introduced in Section \ref{sec:setting} that we are dealing with.

First, we observe that the concatenation of such neural networks is itself simply another neural network.
In particular, if $\varphi_1\in\Phi(L_1, W_1, S_1, B_1)$ and $\varphi_2\in\Phi(L_2, W_2, S_2, B_2)$ are such that $W_{2, L_2+2}=W_{1, 1}$, then $\varphi_1\circ\varphi_2\in\Phi(L_1+L_2+1, W^\mathrm{cat}_{1, 2}, S_1+S_2, B_1\vee B_2)$, where
\[
	W^\mathrm{cat}_{1, 2}
	=\mat{W_{2,1} & W_{2, 2} & \cdots & W_{2, L_2+1} & W_{1, 1} & W_{1, 2} & \cdots & W_{1, L_1+2}}^\top.
\]
In general, if for some $k\in\N$, $\varphi_i\in\Phi(L_i, W_i, S_i, B_i)$ for $i=1,\ldots,k$, then
\[
	\varphi_1\circ\varphi_2\circ\cdots\circ\varphi_k
	\in\Phi\Big(\sum_{i=1}^{k}L_i+k, W^\mathrm{cat}_{[k]}, \sum_{i=1}^{k}S_i, \max_{i\in\{1, \ldots, k\}}B_i\Big),
\]
where $W^\mathrm{cat}_{[k]}$ is defined recursively as above.
Similarly, if $\varphi_1, \varphi_2$ are as before but with $L_1=L_2=L$, we can parallelise the two into one network $\varphi^\mathrm{par}_{1, 2}\in\Phi(L, W^\mathrm{par}_{1, 2}, S_1+S_2, B_1\vee B_2)$ such that $\varphi^\mathrm{par}_{1, 2}(x, y)=\mat{\varphi_1(x) & \varphi_2(y)}^\top$ for $x\in\R^{W_{1, 1}}$ and $y\in\R^{W_{2, 1}}$.
The simplest way to construct this is using block matrices, i.e., by
\[
	\varphi^\mathrm{par}_{1, 2}
	=\mat{A_{1, L} & 0 \\ 0 & A_{2, L}}\sigma_{\mat{b_{1, L} \\ b_{2, L}}}
	\mat{A_{1, L-1} & 0 \\ 0 & A_{2, L-1}}\sigma_{\mat{b_{1, L-1} \\ b_{2, L-1}}}\cdots
	\mat{A_{1, 1} & 0 \\ 0 & A_{2, 1}}\sigma_{\mat{b_{1, 1} \\ b_{2, 1}}}
	\mat{A_{1, 0} & 0 \\ 0 & A_{2, 0}},
\]
which would mean $W^\mathrm{par}_{1, 2}=W_1+W_2$. 
However, if $\varphi_1$ and $\varphi_2$ share some inputs (i.e., if the first, say, $m\in\N$ entries of $x, y$ are the same), then the rightmost matrix in the above may be altered to
\[
	\mat{A_{1, 0} & 0 \\ 0 & A_{2, 0}}\mat{I_m & 0 & 0 \\ 0 & I_{W_{1, 1}-m} & 0 \\ I_m & 0 & 0 \\ 0 & 0 &I_{W_{2, 1}-m}},
\]
whereby $(W^\mathrm{par}_{1, 2})_1=W_{1, 1}+W_{2, 1}-m$ instead.
Again, this can of course naturally be generalised to $k$ networks of equal depth, where we  then have
\[
	\varphi^\mathrm{par}_{[k]}
	\in\Phi\Big(L, W^\mathrm{par}_{[k]}, \sum_{i=1}^{k}S_k, \max_{i\in\{1, \ldots, k\}}B_i\Big),
	\quad (W_{[k]}^\mathrm{par})_j=\sum_{i=1}^{k}W_{i, j}\text{ for }j>1.
\]
Finally, note that multiplying the network $\varphi^{\mathrm{par}}_{[k]}$ with the vector $[1\, \cdots\, 1]$ from the left sums the entries of $\varphi^{\mathrm{par}}_{[k]}$ without changing the size of the network substantially, whence
\[
        \sum_{i=1}^k\varphi_i\in\Phi\Big(L, W^\mathrm{sum}_{[k]}, k+\sum_{i=1}^{k}S_k, 1\vee\max_{i\in\{1, \ldots, k\}}B_i\Big),
	\quad (W_{[k]}^\mathrm{sum})_j=\sum_{i=1}^{k}W_{i, j}\text{ for }1<j<k.
\]
For larger and more complicated neural networks, their exact sizes are often unavailable, and we only have access to their asymptotic sizes.
Due to this, we also introduce the following class of neural networks that eases network size analysis in the proofs that follow,
\[
    \widetilde{\Phi}(\widetilde{L}, \widetilde{W}, \widetilde{S}, \widetilde{B})
    \coloneqq \Big\{\varphi\in\Phi(L, W, S, B): L\lesssim\widetilde{L}, \n{W}_{\infty}\lesssim\widetilde{W}, S\lesssim\widetilde{S}\text{ and }B\lesssim\widetilde{B}\Big\}.
\]
With this notation, we have for arbitrary networks $\varphi_i\in\widetilde{\Phi}(L_i, W_i, S_i, B_i)$, $i=1,2$, that $\varphi_1\circ\varphi_2\in\widetilde{\Phi}(L_1+L_2, W_1\vee W_2, S_1+S_2, B_1\vee B_2)$ and $\varphi_{1, 2}^{\mathrm{par}}\in\widetilde{\Phi}(L_1\vee L_2, W_1+W_2, S_1+S_2, B_1\vee B_2)$.

With this class established, we can begin to approximate $\nabla_x\log h_N$ with a suitable network.
We do this by approaching the problem in smaller pieces, first noting that, for $t\ge0$ and $x\in D$,
\[
	\nabla_x\log h_N(x,t)
	=\frac{\sum_{j=1}^{N}\mathrm{e}^{-t\lambda_j}\langle p_0, e_j\rangle_{L^2}\nabla_xe_j(x)}{\frac{1}{\mathrm{Leb}(D)}+\sum_{j=1}^{N}\mathrm{e}^{-t\lambda_j}\langle p_0, e_j\rangle_{L^2} e_j(x)}.
\]
Thus, in order to approximate $\nabla_x\log h_N$, we need to be able to approximate products and quotients of functions.
Such approximation results already exist in the literature, see, e.g., \cite{oko23, schmidthieber20, telgarsky17, yarotsky17}, but here we give slightly stronger versions with optimised neural network sizes, which lead to improved convergence rates (in terms of $\log$ factors).
More detailed statements and their proofs are given in Appendix \ref{app:neural}, but for our purposes the following will suffice.

\begin{lemma}
	\label{lem:mult_network_asymp}
	For $m\in\N$ and $C\ge1$, there exist neural networks $\varphi_{m}^{\mathrm{mult}}\in\widetilde{\Phi}(m, 1, m, C)$ and $\varphi_{m}^{\mathrm{mult}, d}\in\widetilde{\Phi}(m, d, dm, C)$ satisfying
    \[
        |\varphi^{\mathrm{mult}}_m(x, y)-xy|
        \le C2^{-m},\quad x\in[0, 1],y\in[-C, C],
    \]
    and
    \[
        |\varphi^{\mathrm{mult}, d}_m(x, y)-xy|
        \le\sqrt{d}C2^{-m},\quad x\in[0, 1],y\in[-C, C]^d.
    \]
    These also satisfy $\varphi^{\mathrm{mult}}_m(x, 0)=\varphi^{\mathrm{mult}}_m(0, y)=0$.
\end{lemma}

\begin{lemma}
	\label{lem:rec_network_asymp}
For $m,\underline{k},\overline{k}\in\N$, there exists a neural network $\varphi_{m}^{\mathrm{rec}}\in\widetilde{\Phi}( (k+m)\log(k+m), k, (k+m)\log(k+m), 2^k)$, where $k=\underline{k}+\overline{k}$, satisfying
	\[
		|\varphi_{m}^{\mathrm{rec}}(x)-x^{-1}|
		\le2^{-m},\qquad x\in[2^{-\underline{k}}, 2^{\overline{k}}].
	\]
\end{lemma}

\begin{lemma}
    \label{lem:cap_network}
For each $m\in\N$, there exists a neural network $\varphi^{\mathrm{cap}}\in\widetilde{\Phi}(m\log m, m, m\log m, 2^{m/2})$ satisfying $\varphi^{\mathrm{cap}}(t)\asymp\frac{1}{\sqrt{t}}$ for all $t\in[2^{-m}, 1]$.
\end{lemma}

With these, we are ready to prove the following key approximation result on $h_N$ and $\nabla_xh_N$.
\begin{lemma}
	\label{lemma:approx_h_N}
Let $0<\underline{T}<\overline{T}$ with $\underline{T}\in\mathrm{Poly}(N^{-1})$ be given, and let $f$ denote either $h_N$ or $\partial_{x_i}h_N$, for some $i\in\{1, \ldots, d\}$ and $N\in\N$ sufficiently large.
Then, there exists a neural network $\varphi_{f}\in\widetilde{\Phi}(L(N), W(N), S(N), B(N))$ satisfying
	\[
		\forall t \in [\underline{T}, \overline{T}]: \n{\varphi_{f}(\,\cdot,t)-f(\,\cdot,t)}_{L^2}^2
		\lesssim \begin{cases}
			N^{-\frac{2s}{d}}(\log N)^2, &\text{ if }f=h_N\\
			\varepsilon(t)N^{-\frac{2s}{d}}(\log N)^2, &\text{ if }f=\partial_{x_i}h_N,
		\end{cases}
	\]
	where
	\[
		\int_{\underline{T}}^{\overline{T}}\varepsilon(t)\,\mathrm{d}t
		\lesssim \overline{T}+\log(\underline{T}^{-1})
	\]
	and whose network size is evaluated as 
	\begin{align*}
		L(N)
		&=\log N\log\log N, \\
		W(N)
		&=MN\log N, \\
		S(N)
		&=MN(\log N)^2,\quad\text{and} \\
		B(N)
		&={\color{black}\frac{N^{\frac{2s+d}{2d}}}{\log N}}\vee\frac{1}{\underline{T}},
	\end{align*}
	where $M\in O(\log\frac{\overline{T}}{\underline{T}})$.
        Furthermore,  there exists a constant $C<\infty$ depending only on $p_0$ and $D$ such that $\sup_{x\in D}|\varphi_{\partial_{x_i}h_N}(x, t)|\le C(1\vee\frac{1}{\sqrt{t}})$ for all $t\in[\underline{T},\overline{T}]$. 
\end{lemma}

\begin{proof}
	We first construct a network $\varphi_{f}$ with the desired error rate and specify its size at the end. To this end, suppose that there exist neural networks $\varphi_{f}^{(1)},\ldots,\varphi_{f}^{(M)}$, where $M=\lfloor\log_2\frac{\overline{T}}{\underline{T}}\rfloor$ such that
	\[
		\n{\varphi_{f}^{(m)}(\,\cdot,t)-f(\,\cdot,t)}_{L^2}^2\lesssim \begin{cases}
			N^{-\frac{2s}{d}}(\log N)^2, &\text{ if }f=h_N,\\
			\big(\frac{1}{2^{m-1}\underline{T}}\vee 1\big)N^{-\frac{2s}{d}}(\log N)^2, &\text{ if }f=\partial_{x_i}h_N,
		\end{cases}
	\]
	for $m=1,\ldots,M$ and $t\in[2^{m-1}\underline{T},2^{m+1}\underline{T}]$.
	Then, consider the partition of unity $\{\pi_m\}_{m=1}^{M}$ given by
	\[
		\pi_m(t)
		=0\vee\Big(\frac{t-2^{m-1}\underline{T}}{2^{m-1}\underline{T}}\wedge\frac{2^{m+1}\underline{T}-t}{2^{m}\underline{T}}\Big)
		=\begin{cases}
			\frac{t}{2^{m-1}\underline{T}}-1, &\text{if }t\in[2^{m-1}\underline{T}, 2^{m}\underline{T}]\\
			2-\frac{t}{2^m\underline{T}}, &\text{if }t\in[2^m\underline{T},2^{m+1}\underline{T}]\\
			0, &\text{otherwise}
		\end{cases},
	\]
	for $m=2,\ldots,M-1$, while $\pi_1(t)=0\vee\Big(1\wedge\frac{4\underline{T}-t}{2\underline{T}}\Big)$ and $\pi_{M}(t)=0\vee\Big(1\wedge\frac{t-2^{M-1}\underline{T}}{2^{M-1}\underline{T}}\Big)$.
	Since for $a, b\in\R$, $a\vee b=a+\sigma(b-a)$ and $a\wedge b=a-\sigma(a-b)$, each $\pi_m$ is representable as a neural network.
	We then claim that $\varphi_{f}$, defined as
	\[
		\varphi_{f}(x, t)
		=\sum_{m=1}^{M}\varphi_{\ell_1}^{\mathrm{mult}}\big(\pi_m(t),\varphi_{f}^{(m)}(x, t)\big),\qquad\ell_1=\Big\lceil\frac{s}{d}\log_2 N\Big\rceil,
	\]
	yields the desired network.
	Indeed, we first note that since at most two of the $\pi_m$'s are non-zero for any $t\in[\underline{T}, \overline{T}]$ and $\varphi_{\ell}^{\mathrm{mult}}(0, y)=0$ for all $y\in\R$, we have for $m=2,\ldots,M$ and $t\in[2^{m-1}\underline{T}, 2^{m}\underline{T}]$ that
	\begin{align*}
		\n{\varphi_{f}(\,\cdot,t)-f(\,\cdot,t)}_{L^2}
		&=\n{\varphi_{\ell_1}^{\mathrm{mult}}\big(\pi_{m-1}(t), \varphi_{f}^{(m-1)}(\,\cdot, t)\big)+\varphi_{\ell_1}^{\mathrm{mult}}\big(\pi_{m}(t), \varphi_{f}^{(m)}(\,\cdot, t)\big)-f(\,\cdot,t)}_{L^2} \\
		&\lesssim 2^{-\ell_1}+\n{\pi_{m-1}(t)\varphi_{f}^{(m-1)}(\,\cdot,t)+\pi_{m}(t)\varphi_{f}^{(m)}(\,\cdot,t)-f(\,\cdot,t)}_{L^2} \\
		&\lesssim \begin{cases}
			2^{-\ell_1}+N^{-\frac{s}{d}}\log N, &\text{ if }f=h_N\\
			2^{-\ell_1}+\big(\frac{1}{\sqrt{2^{m-2}\underline{T}}}\vee1\big)N^{-\frac{s}{d}}\log N, &\text{ if }f=\partial_{x_i}h_N,
		\end{cases}
	\end{align*}
	where in the last inequality we used
	\begin{align*}
		\n{\pi_{m-1}(t)&\varphi_{f}^{(m-1)}(\,\cdot,t)+\pi_{m}(t)\varphi_{f}^{(m)}(\,\cdot,t)-f(\,\cdot,t)}_{L^2} \\
		&=\n{\pi_{m-1}(t)\big(\varphi_{f}^{(m-1)}(\,\cdot,t)-f(\,\cdot,t)\big)+\pi_{m}(t)\big(\varphi_{f}^{(m)}(\,\cdot,t)-f(\,\cdot,t)\big)}_{L^2} \\
		&\le\pi_{m-1}(t)\n{\varphi_{f}^{(m-1)}(\,\cdot,t)-f(\,\cdot,t)}_{L^2}+\pi_{m}(t)\n{\varphi_{f}^{(m)}(\,\cdot,t)-f(\,\cdot,t)}_{L^2} \\
		&\lesssim \begin{cases}
			N^{-\frac{s}{d}}\log N, &\text{ if }f=h_N\\
			\big(\frac{1}{\sqrt{2^{m-2}\underline{T}}}\vee1\big)N^{-\frac{s}{d}}\log N, &\text{ if }f=\partial_{x_i}h_N,
		\end{cases}.
	\end{align*}
	Setting
	\[
		\varepsilon(t)
		=\sum_{m=1}^{M+1}\Big(\frac{1}{2^{m-2}\underline{T}}\vee 1\Big)\bm{1}_{[2^{m-1}\underline{T}, 2^m\underline{T}]}(t),
	\]
	and by choice of $\ell_1$, squaring both sides of the inequality yields the desired error rate.
	A similar but simpler analysis shows that this also holds for $t\in[\underline{T}, 2\underline{T}]$ and $t\in[2^{M}\underline{T}, 2^{M+1}\underline{T}]$, whence the inequality holds for all $t\in[\underline{T}, 2^{M+1}\underline{T}]\supseteq[\underline{T}, \overline{T}]$.
	Furthermore, we have
	\begin{align*}
		\int_{\underline{T}}^{\overline{T}}\varepsilon(t)\,\mathrm{d}t
		&\le\sum_{m=1}^{M+1}\Big(\frac{1}{2^{m-2}\underline{T}}\vee 1\Big)(2^m\underline{T}-2^{m-1}\underline{T}) \\
		&=\sum_{m=1}^{\lfloor\log_2(\underline{T}^{-1})\rfloor+2}\frac{2^{m-1}\underline{T}}{2^{m-2}\underline{T}}+\sum_{m=\lfloor\log_2(\underline{T}^{-1})\rfloor+3}^{M+1}2^{m-1}\underline{T} \\
		&=2\big(\big\lfloor\log_2(\underline{T}^{-1})\big\rfloor+2\big)+\underline{T}\Big(\sum_{m=0}^{M}2^m-\sum_{m=0}^{\lfloor\log_2(\underline{T}^{-1})\rfloor+2}2^m\Big) \\
		&=2\big(\big\lfloor\log_2(\underline{T}^{-1})\big\rfloor+2\big)+2\underline{T}\Big(2^M-2^{\lfloor\log_2(\underline{T}^{-1})\rfloor}\Big) \\
		&\lesssim\overline{T}+\log(\underline{T}^{-1})
	\end{align*}
	as claimed.
  
	As such, we only need to construct the networks $\varphi_{f}^{(m)}$ for all $m\in\{1,\ldots,M\}$, so let some such $m$ be fixed.
	Then, let $a_m=3\cdot2^{m-2}\underline{T}$ and $b_m=5\cdot2^{m-2}\underline{T}$, and set $f_m(x, t)=f(x, a_m t+b_m)$ such that $f_m(x, [-1, 1])=f(x, [2^{m-1}\underline{T}, 2^{m+1}\underline{T}])$ for all $x\in D$.
	As in the proof of Lemma \ref{lem:bound31}, we see that for each fixed $t\in[\underline{T}, \overline{T}]$, $h_N(\,\cdot,t)\in H^{s+1}(D)$, whence $f(\,\cdot,t)\in H^s(D)=B_{2, 2}^s(D)$.
	Furthermore, $\lVert f(\,\cdot,t)/(1\vee\n{f(\,\cdot,t)}_{B_{2, 2}^s(D)})\rVert_{B_{2, 2}^s(D)} \leq 1$.
	Thus, since $D$ is bounded, a slight modification of \citet[Proposition 1]{suzuki19} using the Sobolev extension theorem yields the existence of a neural network $\widetilde{\varphi}_{f, t}\in\widetilde{\Phi}(\log N, N, N\log N, N^{1/d})$ satisfying $\n{\widetilde{\varphi}_{f, t}-f(\,\cdot,t)/(1\vee\n{f(\,\cdot,t)}_{B_{2, 2}^s(D)})}_{L^2(D)}\lesssim N^{-\frac{s}{d}}$.
	Then, setting $\varphi_{f, t}=(1\vee\n{f(\,\cdot,t)}_{B_{2, 2}^s(D)})\widetilde{\varphi}_{f, t}$, we have
	\begin{align*}
		\n{\varphi_{f, t}-f(\,\cdot,t)}_{L^2(D)}
		&=(1\vee\n{f(\,\cdot,t)}_{B_{2, 2}^s(D)})\n[\Big]{\widetilde{\varphi}_{f, t}-\frac{f(\,\cdot,t)}{(1\vee\n{f(\,\cdot,t)}_{B_{2, 2}^s(D)})}}_{L^2(D)}\\
		&\lesssim(1\vee\n{f(\,\cdot,t)}_{B_{2, 2}^s(D)})N^{-\frac{s}{d}}.
	\end{align*}
	Noting that
	\[
		\n{f(\,\cdot,t)}_{B_{2, 2}^s(D)}
		\asymp\n{f(\,\cdot,t)}_{H^s(D)}
		\lesssim \begin{cases}
			\n{p_0}_{H^s}, &\text{ if }f=h_N,\\
			\frac{1}{\sqrt{t}}\n{p_0}_{H^s}, &\text{ if }f=\partial_{x_i}h_N,
		\end{cases}
	\]
	we thus have that $\varphi_{f, t}\in\widetilde{\Phi}(\log N, N, N\log N, N^{1/d}\vee\frac{1}{\sqrt{\underline{T}}})$, and
	\[
		\n{\varphi_{f, t}-f(\,\cdot,t)}_{L^2(D)}
		\lesssim \begin{cases}
			N^{-\frac{s}{d}}, &\text{ if }f=h_N,\\
			\big(\frac{1}{\sqrt{t}}\vee1\big)N^{-\frac{s}{d}}, &\text{ if }f=\partial_{x_i}h_N.
		\end{cases}
	\]
	Then, for each $t\in[-1, 1]$, let $\varphi_{f_m, t}=\varphi_{f, a_mt+b_m}$ such that $\varphi_{f_m, t}$ is an approximation of $f_m(t, \cdot)$.
	We will then approximate $f_m$ by polynomial interpolation in time and by $\varphi_{f_m, t}$ in space.
	There are then three main sources of error: the error in approximating a polynomial with a neural network, the error from polynomial interpolation, and finally the error from approximating $f_m(\,\cdot,t)$ by $\varphi_{f_m, t}$.
	To separate these sources of error, we now let $\{t_i\}_{i=0}^k$ be the first $k+1$ Chebyshev nodes on $[-1, 1]$ for some $k$ to be determined later, i.e. $t_i=\cos\frac{i\pi}{k}$.
	Then, for $i=0,\ldots,k$, let $p_i(t)=\prod_{j\neq i}(t-t_j)$ and set $c_i=\frac{1}{p_i(t_i)}$.
	Furthermore, set
	\begin{align*}
		\varphi_{f_m}(x, t)
		&=\sum_{i=0}^{k}c_i\varphi_{\ell_2}^{\mathrm{mult}}\big(\varphi_{p_i}(t), \varphi_{f_m, t_i}(x)\big), \\
		\psi_{m}(x, t)
		&=\sum_{i=0}^{k}c_ip_i(t)\varphi_{f_m, t_i}(x), \\
		P_m(x, t)
		&=\sum_{i=0}^{k}c_ip_i(t)f_m(x, t_i),
	\end{align*}
	where $\varphi_{p_i}$ is a neural network approximation of $p_i$ satisfying $|\varphi_{p_i}(t)-p_i(t)|\lesssim k2^{-\ell_3}$ to be constructed later.
	We then have that
	\begin{equation}
		\label{eq:fm_error}
        \begin{split}
		\n{\varphi_{f_m}(\,\cdot,t)-f_m(\,\cdot,t)}_{L^2}
		&\le\n{\varphi_{f_m}(\,\cdot,t)-\psi_m(\,\cdot,t)}_{L^2}
		+\n{\psi_m(\,\cdot,t)-P_m(\,\cdot,t)}_{L^2}\\
        &\quad +\n{P_m(\,\cdot,t)-f_m(\,\cdot,t)}_{L^2}.
        \end{split}
	\end{equation}
	By \citet[Theorem 5.2]{trefethen13}, it holds that $|c_i|\le \frac{2^{k-1}}{k}$, and so the first term of \eqref{eq:fm_error} is upper bounded by
	\begin{align*}
		\n{\varphi_{f_m}(\,\cdot,t)-\psi_m(\,\cdot,t)}_{L^2}
		&\le\sum_{i=0}^{k}|c_i|\n{\varphi_{\ell_2}^{\mathrm{mult}}\big(\varphi_{p_i}(t), \varphi_{f_m, t_i}\big)-p_i(t)\varphi_{f_m, t_i}}_{L^2} \\
		&\lesssim\sum_{i=0}^{k}|c_i|(2^{-\ell_2}+k2^{-\ell_3}\n{\varphi_{f_m, t_i}}_{L^2}) \\
		&\lesssim\sum_{i=0}^{k{}}|c_i|(2^{-\ell_2}+k2^{-\ell_3}(\underline{T}^{-\frac{1}{2}}N^{-\frac{s}{d}}+\n{f_m(\,\cdot,t_i)}_{L^2})) \\
		&\le 2^{k-1}(2^{-\ell_2}+k2^{-\ell_3}(\underline{T}^{-\frac{1}{2}}N^{-\frac{s}{d}}+\n{p_0}_{H^1})),
	\end{align*}
	and choosing $\ell_2=\lceil\frac{s}{d}\log_2 N+k\rceil$ and $\ell_3=\ell_2+\lceil\log_2 (k+\underline{T}^{-\frac{1}{2}     })\rceil$ bounds this term by $N^{-s/d}$.
	For the second term of \eqref{eq:fm_error}, it is a well-known property of Chebyshev nodes that $|p_i(t)c_i|{\color{black}\le 2}$, whence
	\begin{align*}
		\n{\psi_m(\,\cdot,t)-P_m(\,\cdot,t)}_{L^2}
		&\le\sum_{i=0}^{k}|c_ip_i(t)|\n{\varphi_{f_m, t_i}-f_m(\,\cdot,t_i)}_{L^2}\\
		&\lesssim \begin{cases}
			kN^{-\frac{s}{d}}, &\text{ if }f=h_N\\
			k\big(\frac{1}{\sqrt{2^{m-1}\underline{T}}}\vee1\big)N^{-\frac{s}{d}}, &\text{ if }f=\partial_{x_i}h_N,
		\end{cases}
	\end{align*}
	where the last inequality in the case of $f=\partial_{x_i}h_N$ follows from the fact that for all $t\in[-1, 1]$
	\begin{align*}
		\n{\varphi_{f_m, t}-f_m(\,\cdot, t)}_{L^2}
		&=\n{\varphi_{f,a_mt+b_m}-f(\,\cdot, a_mt+b_m)}_{L^2} \\
		&\lesssim\Big(\frac{1}{\sqrt{a_mt+b_m}}\vee1\Big)N^{-\frac{s}{d}} \\
		&\le\Big(\frac{1}{\sqrt{2^{m-1}\underline{T}}}\vee1\Big)N^{-\frac{s}{d}}.
	\end{align*}
	Finally, for the third term of \eqref{eq:fm_error}, we note that for each $x\in D$, the function $t\mapsto f_m(x, t)$ is entire on $\C$ as an affine combination of exponentials.
	It then follows from \citet[Theorem 8.2]{trefethen13} that, for $\rho>1$,
	\[
		|f_m(x, t)-P_m(x, t)|
		\le\frac{4M_{m, \rho}(x)\rho^{-k}}{\rho-1},\qquad t\in[-1, 1],
	\]
	where
	\[
		M_{m, \rho}(x)
		=\max_{z\in \partial E_\rho}|f_m(x, z)|,\quad\text{and}\quad
		\partial E_\rho
		=\left\{\frac{z+z^{-1}}{2}\mid |z|=\rho\right\}.
	\]
	For $z\in\partial E_\rho$, we have, letting $\widetilde{e}_n$ denote either $e_n$ or $\partial_{x_i}e_n$, depending on $f$,
	\begin{align*}
		f_m(x, z)
		&=f(x, a_mz+b_m)
		=\sum_{n=0}^{N}\mathrm{e}^{-\lambda_n(a_m\Re{z}+b_m)}\langle p_0, e_n\rangle \widetilde{e}_n(x)\mathrm{e}^{-\mathrm{i}\lambda_n a_m\Im z}
		=\langle \bm{r}(x, \Re z), \bm{\theta}(\Im z)\rangle,
	\end{align*}
	where $(\bm{r}(x,  y))_{n}=\mathrm{e}^{-\lambda_n(a_my+b_m)}\langle p_0, e_n\rangle \widetilde{e}_n(x)$ and $\bm{\theta}(y)_{n}=\mathrm{e}^{-\mathrm{i}\lambda_n a_m y}$, so $\lvert \bm{\theta}(y)\rvert = N+1$.
	Thus, by Cauchy--Schwarz inequality,
	\begin{align*}
		M_{m, \rho}(x)
		&\le \bigg((N+1)\max_{z\in\partial E_\rho}\sum_{n=0}^{N}\mathrm{e}^{-2\lambda_n(a_m\Re{z}+b_m)}\langle p_0, e_n\rangle^2\widetilde{e}_n(x)^2\bigg)^{1/2} \\
		&=\bigg((N+1)\sum_{n=0}^{N}\mathrm{e}^{-2\lambda_n(a_m(\frac{-\rho-\rho^{-1}}{2})+b_m)}\langle p_0, e_n\rangle^2\widetilde{e}_n(x)^2\bigg)^{1/2}.
	\end{align*}
	Consequently,
	\begin{align*}
		\n{M_{m,\rho}}_{L^2}^2
		&\le(N+1)\sum_{n=0}^{N}\mathrm{e}^{-2\lambda_n(a_m(\frac{-\rho-\rho^{-1}}{2})+b_m)}\langle p_0, e_n\rangle^2\n{\widetilde{e}_n}_{L^2}^2 \\
            &\le(N+1)\sum_{n=0}^{N}\mathrm{e}^{-2\lambda_n(a_m(\frac{-\rho-\rho^{-1}}{2})+b_m)}\langle p_0, e_n\rangle^2\n{e_n}_{H^1}^2 \\
		&\lesssim(N+1)\sum_{n=0}^{N}\mathrm{e}^{-2\lambda_n(a_m(\frac{-\rho-\rho^{-1}}{2})+b_m)}\lambda_n\langle p_0, e_n\rangle^2,
	\end{align*}
	and if $a_m(\frac{-\rho-\rho^{-1}}{2})+b_m=0$, the right hand side is bounded by $(N+1)\n{p_0}_{\bar{H}^1}^2$.
	This is exactly the case when $\rho=3$, in which case we have
	\[
		\n{f_m(\,\cdot,t)-P_m(\,\cdot,t)}_{L^2}^2
		\lesssim N3^{-2k},
	\]
	and setting $k=\lceil(\frac{s}{d}+\frac{1}{2})\log N\rceil$ yields an approximation $\varphi_{f_m}$ of $f_m$ with the correct error rate.
	Setting $\varphi_{f}^{(m)}(x, t)=\varphi_{f_m}(\frac{t-b_m}{a_m}, x)$ for $t\in[2^{m-2}\underline{T}, 2^{m+2}\underline{T}]$ then gives the desired network.
	We thus only need to construct the network $\varphi_{p_i}$ as detailed above.
	To this end, note that the neural network
	\[
		\varphi_{p_i}^{(0)}\colon
		t\mapsto\mat{I_{k} & -I_{k} & 0 \\ 0 & 0 & I_{2^{\lceil\log_2k\rceil}-k}}\sigma_{\mat{\bm{t}_i \\ -\bm{t}_i \\ -1}}\mat{I_{k} \\ -I_{k} \\ 0}t,\qquad
		\bm{t}_i
		=\mat{t_0 & \cdots & t_{i-1} & t_{i+1} & \cdots & t_k}^\top,
	\]
	maps $t$ to the vector 
    \[\mat{t-t_0 & \cdots & t-t_{i-1} & t-t_{i+1} & \cdots & t-t_k & 1 & \cdots & 1}^\top\in[-2, 2]^{2^{\lceil\log_2k\rceil}},\] 
    regardless of the signs of its entries.
	Without altering the size of $\varphi_{p_i}^{(0)}$, we can swap the entries of $\varphi_{p_i}^{(0)}(t)$ such that adjacent entries correspond to opposing $t_j$'s, e.g., such that the first two entries are $t-t_0$ and $t-t_k$ rather than $t-t_0$ and $t-t_1$ and so on.
	This ensures that the products of adjacent entries stay uniformly bounded rather than some products growing and some shrinking, better bounding the size of the network.
	A slight modification of (the proof of) Lemma \ref{lemma:mult_network} then yields a network $\overline{\varphi}^{\mathrm{mult}}_{\ell_3}$ of the same asymptotic size as in Lemma \ref{lemma:mult_network} such that $|\varphi^{\mathrm{mult}}_{\ell_3}(x, y)-xy|\le 2^{-\ell_3}$ for all $x, y\in[-2, 2]$.
	Let $\varphi_{p_i}^{(j)}$ be a parallelisation of $2^{\lceil\log_2k\rceil-j}$ copies of this network, and set
	\[
		\varphi_{p_i}
		=\varphi_{p_i}^{(\lceil\log_2k\rceil)}\circ\cdots\circ\varphi_{p_i}^{(1)}\circ\varphi_{p_i}^{(0)}.
	\]
	To ensure that $\varphi_{p_i}$ satisfies $|\varphi_{p_i}-p_i|\lesssim k2^{-\ell_3}$, fix $t\in[-1, 1]$, and for notation, set $\xi^{(j)}_n=(\varphi_{p_i}^{(j)}\circ\cdots\circ\varphi_{p_i}^{(1)}\circ\varphi_{p_i}^{(0)}(t))_n$ and $y_l=(\varphi_{p_i}^{(0)}(t))_l$.
	We then claim that
	\[
		\Big|\xi^{(j)}_{n}-\prod_{l=(n-1)2^j+1}^{n2^j}y_l\Big|
		\le (2^j-1)2^{-\ell_3}.
	\]
	This is true by construction for $j=1$, so assume this holds for some $j\ge1$.
	We then have
	\begin{align*}
		\bigg|\xi^{(j+1)}_{n}-\prod_{l=(n-1)2^{j+1}+1}^{n2^{j+1}}y_l\bigg|
		&=\bigg|\varphi^{\mathrm{mult}}_{\ell_3}(\xi^{(j)}_{2n-1}, \xi^{(j)}_{2n})-\prod_{l=(n-1)2^{j+1}+1}^{n2^{j+1}}y_l\bigg| \\
		&\le 2^{-\ell_3}+\bigg|\xi_{2n-1}^{(j)}\xi_{2n}^{(j)}-\bigg(\prod_{l=(2n-2)2^j+1}^{(2n-1)2^j}y_l\bigg)\bigg(\prod_{l=(2n-1)2^j+1}^{2n2^j}y_l\bigg)\bigg| \\
		&\le2^{-\ell_3}\bigg(1+\bigg(|\xi_{2n-1}^{(j)}|+\prod_{l=(2n-1)2^j+1}^{2n2^j}|y_l|\bigg)(2^{j}-1)\bigg).
	\end{align*}
	By our previous rearranging of entries in $\varphi_{p_i}^{(0)}(t)$, we have that $|\xi_{2n-1}^{(j)}|+\prod_{l=(2n-1)2^j+1}^{2n2^j}|y_l|\le 2$ for all $j\ge1$, and the claim follows.
	This then implies that $|\varphi_{p_i}(t)-p_i(t)|\lesssim k2^{-\ell_3}$ for all $t\in[-1, 1]$ as desired.

	We now shift from analysing the error of the network to its size instead.
	First, it is apparent from their construction that $\varphi_{p_i}^{(0)}\in\widetilde{\Phi}(1, k, k, 1)$, while it also holds that $\varphi_{p_i}^{(j)}\in\widetilde{\Phi}(\ell_3, 2^{\lceil\log_2k\rceil-(j-1)}, 2^{\lceil\log_2k\rceil-j}\ell_3, 1)$.
	Hence, since $\sum_{j=1}^{\lceil\log_2k\rceil}2^{\lceil\log_2k\rceil-j}=\sum_{j=0}^{\lceil\log_2k\rceil-1}2^j=2^{\lceil\log_2k\rceil}-1$,
	\begin{align*}
		\varphi_{p_i}
		&\in\widetilde{\Phi}\Big(\lceil\log_2k\rceil\ell_3, k, (2^{\lceil\log_2k\rceil}-1)\ell_3+k, 1\Big) \\
		&=\widetilde{\Phi}(\log N\log\log N, \log N, (\log N)^2, 1).
	\end{align*}
	Parallelising this with $\varphi_{f_m, t_i}$, we thus get a network in $\widetilde{\Phi}(\log N\log\log N, N, N\log N, N^{\frac{1}{d}}\vee\frac{1}{\sqrt{\underline{T}}})$.
	Since this dominates the size of $\varphi_{\ell_2}^{\mathrm{mult}}$ and since $|c_i|{\color{black}\le\frac{2^{k-1}}{k}\le\frac{N^{\frac{2s+d}{2d}}}{\frac{2s+d}{d}\log N}}$, it follows that each term of $\varphi_{f_m}$ {\color{black} is in $\widetilde{\Phi}(L_M, W_M, S_M, B_M)$, where}
	\begin{align*}
		L_M
		&=\log N\log\log N, \\
		W_M
		&=N\log N, \\
		S_M
		&=N(\log N)^2,\quad\text{and} \\
		B_M
		&={\color{black}\frac{N^{\frac{2s+d}{2d}}}{\log N}}\vee\frac{1}{\sqrt{\underline{T}}}.
	\end{align*}
	Next, by construction of the $\pi_m$'s, we have that $\pi_m\in\widetilde{\Phi}(1, 1, 1, \frac{1}{\underline{T}})$, where parallelising $\varphi_{f}^{(m)}$ with $\pi_m$ does not change the asymptotic size of the network.
	Since the size of $\varphi_{\ell_1}^{\mathrm{mult}}$ is also dominated by that of $\varphi_{f_m}$, it follows that each term of $\varphi_{f}$ is included in $\widetilde{\Phi}(L_M, W_M, S_M, B_M)$ as well. Parallelising all of these and summing them yields
	\begin{align*}
		L(N)
		&=\log N\log\log N, \\
		W(N)
		&=MN\log N, \\
		S(N)
		&=MN(\log N)^2,\quad\text{and} \\
		B(N)
		&={\color{black}\frac{N^{\frac{2s+d}{2d}}}{\log N}}\vee\frac{1}{\underline{T}}.
	\end{align*}
    Finally, since $s>\frac{d}{2}$, we have by similar calculations as those in the proof of Lemma \ref{prop:approx_spectral0} that for some $\widetilde{C}<\infty$ depending only on $D$ and $p_0$, it holds that $\sup_{x\in D}|\partial_{x_i}h_N(x, t)|\le\widetilde{C}\frac{1}{\sqrt{t}}$.
    Letting $\varphi^{\mathrm{cap}}$ be as in Lemma \ref{lem:cap_network} (with $m=\log_2(\underline{T}^{-1})\asymp\log N$), it follows that also $\sup_{x\in D}|\partial_{x_i}h_N(x, t)|\le C\varphi^{\mathrm{cap}}(t)$ for all $t\in[\underline{T},\overline{T}]$, whence we get a no worse approximation by replacing $\varphi_{\partial_{x_i}h_N}$ with $\varphi_{\partial_{x_i}h_N}\wedge (C\varphi^{\mathrm{cap}})$ and this network has the desired bound.
    Furthermore, since $\varphi^{\mathrm{cap}}\in\widetilde{\Phi}(\log N\log\log N, \log N, \log N\log\log N, 1/\sqrt{\underline{T}})$, and this is dominated by the network size of $\varphi_{\partial_{x_i}h_N}$, taking this minimum does not alter the size of the network.
    This finishes the proof.
\end{proof}

\paragraph{Step 3: Putting things together}
With the essential preparations from Step 1 and 2, we can now finally prove Theorem \ref{theorem:approx_nabla_log_h_N}.
\medskip 

\begin{proof}[Proof of Theorem \ref{theorem:approx_nabla_log_h_N}]
Let $\varphi_{h_N},\varphi_{\partial_{x_i} h_N}$, $i\in\{1, \ldots, d\}$ and $N\in\N$, be as in Lemma \ref{lemma:approx_h_N}.
	Parallelising the latter of these yields a network $\varphi_{\nabla_xh_N}\in\widetilde{\Phi}(L(N), W(N), S(N), B(N))$ approximating $\nabla_xh_N$.
	We then claim that $\varphi_{\sco}$ defined as
	\[
		\varphi_{\sco}(x, t)
		=\varphi^{\mathrm{mult}, d}_{\ell}(\varphi^{\mathrm{rec}}_{\ell}(\varphi_{h_N}(x, t)\vee\alpha), \varphi_{\nabla_xh_N}(x, t)),\qquad
		\ell
		=\Big\lceil\frac{s}{d}\log_2N\Big\rceil,
	\]
	has the desired properties.
	Indeed, for the size of the network, notice that $\varphi_{h_N}\vee\alpha$ is bounded above (by $\n{p_0}_{\infty}$) and below, and hence $\varphi^{\mathrm{rec}}_{\ell}\circ(\varphi_{h_N}\vee\alpha)$ is bounded above by $2\alpha^{-1}$ for $N$ large enough, while the entries of $\varphi_{\nabla_x h_N}$ are all bounded numerically by $\frac{C}{\sqrt{\underline{T}}}$ some $C<\infty$ since $s>\frac{d}{2}$.
	Thus, $\varphi_{\ell}^\mathrm{rec}\in\widetilde{\Phi}(\log N\log\log N, \log N\log\log N, \log N\log\log N, 1)$ and $\varphi^{\mathrm{mult},d}_{\ell}\in\widetilde{\Phi}(\log N, \log N, \log N, \frac{1}{\sqrt{\underline{T}}})$, whereby $\varphi_{\sco}$ has the same asymptotic size as $\varphi_{\nabla_xh_N}$.
        Note also that $|\varphi_\sco(x, t)|\le \frac{4C\alpha^{-1}}{\sqrt{t}}\vee1$ for all $x\in D$, $t\in[\underline{T},\overline{T}]$ and $N$ large enough.
So all that remains is to show that $\varphi_\sco$, as defined above, satisfies \eqref{eq:err_rate_score_approximation}.
	By Proposition \ref{prop:approx_spectral0} and the triangle inequality, this is equivalent to verifying
	\begin{equation*}
		\int_{\underline{T}}^{\overline{T}}\int_D
		\bigg|\varphi_{\sco}(x,t)-\frac{\nabla_xh_N(x,t)}{h_N(x,t)\vee\alpha}\bigg|^2p_t(x)\,\mathrm{d}x\,\mathrm{d}t
		\lesssim N^{-\frac{2s}{d}}(\log N)^2(\overline{T}+\log(\underline{T}^{-1})).
	\end{equation*}
	which follows if we can show that $\n{|\varphi_\sco(\,\cdot,t)-\frac{\nabla_x h_N(\cdot, t)}{h_N(\cdot, t)\vee\alpha}|}_{L^2}^2\lesssim \varepsilon(t)N^{-\frac{2s}{d}}(\log N)^2$ for each $t\in[\underline{T},\overline{T}]$, where $\varepsilon(t)$ is as in Lemma \ref{lemma:approx_h_N}.
For doing so, first note that
	\begin{align*}
		\n[\bigg]{\bigg|\varphi_\sco(\,\cdot,t)-\frac{\nabla_x h_N(\cdot, t)}{h_N(\cdot, t)\vee\alpha}\bigg|}_{L^2}
		&\le\n{|\varphi_\sco(\,\cdot,t)-\varphi^{\mathrm{rec}}_{\ell}(\varphi_{h_N}(\,\cdot,t)\vee\alpha)\varphi_{\nabla_xh_N}(\,\cdot,t)|}_{L^2} \\
		&\qquad+\n[\bigg]{\bigg(\varphi_{\ell}^{\mathrm{rec}}(\varphi_{h_N}(\,\cdot,t)\vee\alpha)-\frac{1}{\varphi_{h_N}(\,\cdot,t)\vee\alpha}\bigg)|\varphi_{\nabla_xh_N}(\,\cdot,t)|}_{L^2} \\
		&\qquad+\n[\bigg]{\bigg|\frac{\varphi_{\nabla_xh_N(\cdot, t)}}{\varphi_{h_N}(\,\cdot,t)\vee\alpha}-\frac{\nabla_xh_N(\cdot, t)}{h_N(\cdot, t)\vee\alpha}\bigg|}_{L^2}.
	\end{align*}
	The first term is simply evaluated as
	\[
		\n{|\varphi_\sco(\,\cdot,t)-\varphi^{\mathrm{rec}}_{\ell}(\varphi_{h_N}(\,\cdot,t)\vee\alpha)\varphi_{\nabla_xh_N}(\,\cdot,t)|}_{L^2}
		\le 2^{-\ell}\Big(4\sqrt{d}\alpha^{-1}\n[\Big]{p_0-\frac{1}{\mathrm{Leb}(D)}}_{H^s}\mathrm{Leb}(D)\Big)
		\lesssim N^{-\frac{s}{d}},
	\]
	while for the second term, we have
	\begin{align*}
		\n[\bigg]{\bigg(\varphi_{\ell}^{\mathrm{rec}}(\varphi_{h_N}(\,\cdot,t)\vee\alpha)-\frac{1}{\varphi_{h_N}(\,\cdot,t)\vee\alpha}\bigg)|\varphi_{\nabla_xh_N}(\,\cdot,t)|}_{L^2}
		&\le 2^{-\ell}\n{|\varphi_{\nabla_xh_N}(\,\cdot,t)|}_{L^2} \\
		&\lesssim2^{-\ell}\Big(\sqrt{d\varepsilon(t)}N^{-\frac{s}{d}}\log N+\n{p_0}_{H^1}\Big) \\
		&\lesssim \sqrt{\varepsilon(t)}N^{-\frac{s}{d}}.
	\end{align*}
For the final term, one obtains, similarly to the proof of Proposition \ref{prop:approx_spectral0},
	\begin{align*}
		\n[\bigg]{\bigg|\frac{\varphi_{\nabla_xh_N}(\cdot, t)}{\varphi_{h_N}(\,\cdot,t)\vee\alpha}-\frac{\nabla_xh_N(\cdot, t)}{h_N(\cdot, t)\vee\alpha}\bigg|}_{L^2}
		&\le\alpha^{-1}\n{|\varphi_{\nabla_xh_N}(\,\cdot,t)-\nabla_xh_N(\cdot, t)|}_{L^2} \\
		&\qquad+\alpha^{-2}\n{|\nabla_xh_N(\cdot, t)(h_N(\cdot, t)-\varphi_{h_N}(\,\cdot,t))}_{L^2} \\
		&\lesssim\alpha^{-1}\sqrt{d\varepsilon(t)}N^{-\frac{s}{d}}\log N
		+\alpha^{-2}\n{p_0}_{H^1}\sqrt{\varepsilon(t)}N^{-\frac{s}{d}}\log N \\
		&\lesssim \sqrt{\varepsilon(t)}N^{-\frac{s}{d}}\log N.
	\end{align*}
        Finally, setting $N=n^{\frac{d}{2s+d}}$ yields both the desired network size and error rate,	which finishes the proof.
\end{proof}

\subsection{Proof of the main result}\label{susbsec:proof_main}
With the previous preparations, we can now prove our main result on the generative error, Theorem \ref{theo:main}. To this end, we use the general error decomposition \eqref{eq:error_decomp} in combination with Lemma \ref{lem:early_stop} and Lemma \ref{lem:ergodic_err} to control the early stopping and ergodic error contributions, as well as with the approximation result Theorem \ref{theorem:approx_nabla_log_h_N} that allows us to obtain an optimised upper bound on the empirical score loss via Theorem \ref{theo:score_loss_emp}.
\medskip

\begin{proof}[Proof of Theorem \ref{theo:main}]
Choose $\delta = n^{-2s/(2s+d)}$ and $N = n^{d/(2s+d)}$. By the choices for $\underline{T}, \overline{T}$ and $N$, Theorem \ref{theorem:approx_nabla_log_h_N} implies that there exists a family of neural networks $\mathcal{S}$ with the specified size constraints, such that for some $\mathfrak{s} \in \mathcal{S}$ we have 
\[\int_{\underline{T}}^{\overline{T}}\int_D
\big|\sco(x,t)-\nabla_x\log p_t(x)\big|^2p_t(x)\,\mathrm{d}x\,\mathrm{d}t
\lesssim n^{-\frac{2s}{2s+d}}(\log n)^3.\]
With these network size constraints, we get for $C(\mathcal{S}) \coloneq C$ and $c$ from Lemma \ref{lem:cover}, by using a straightforward modification of \citet[Lemma C.2]{oko23}, that 
\begin{align*} 
\log \mathcal{N}(\mathcal{S},\lVert  \cdot  \rVert_{D \times [\underline{T},\overline{T}]}, \delta/(cC\overline{T})) &\lesssim S(n)L(n)\log\big(\delta^{-1} \overline{T}^2L(n)\lVert W(n) \rVert_\infty B(n)\big)\\
&\lesssim n^{\frac{d}{2s+d}} (\log n)^5 \log \log n.
\end{align*}
Lemma \ref{lem:cover} therefore implies that 
\[\log \mathcal{N}(\mathcal{L},\lVert  \cdot  \rVert_D,\delta) \lesssim n^{\frac{d}{2s+d}} (\log n)^5 \log \log n\]
as well. By Lemma \ref{lem:scoreloss_bound} and the choices of $\overline{T},\underline{T}$, we can choose $C(\mathcal{L}) \lesssim \log n$ so that 
\[\frac{C(\mathcal{L})}{n} \log \mathcal{N}(\mathcal{L},\lVert \cdot \rVert_D,\delta) \lesssim n^{\frac{-2s}{2s+d}} (\log n)^6 \log \log n.\]
Using Theorem \ref{theorem:approx_nabla_log_h_N}, it follows from the above and Theorem \ref{theo:score_loss_emp} by our choice of $N$ that 
\begin{align*} 
&\E\Big[\int_{\underline{T}}^{\overline{T}} \int_{\overline{D}} \lvert \hat{\mathfrak{s}}_n(x,t) - \nabla \log p_t(x) \rvert^2 p_t(x) \diff{x} \diff{t}  \Big] \\ 
&\quad \leq 2 \inf_{\sco \in \mathcal{S}} \int_{\underline{T}}^{\overline{T}} \int_{D} \lvert \sco(x,t) - \nabla \log p_t(x) \rvert^2 p_t(x) \diff{x} \diff{t} + 2\frac{C(\mathcal{L})}{n}\Big(\frac{37}{9}\log \mathcal{N}(\mathcal{L}, \lVert \cdot \rVert_{D}, \delta) + 32\Big) + 3\delta\\
&\quad\lesssim n^{-\frac{2s}{2s+d}}(\log n)^3 + n^{-\frac{2s}{2s+d}} (\log n)^6 \log \log n +  n^{-\frac{2s}{2s+d}}\\
&\quad\lesssim n^{-\frac{2s}{2s+d}} (\log n)^6 \log \log n.
\end{align*}
Thus, 
\begin{align*}
\E\Big[\TV(\overline{p}{}^{\mathfrak{s}^\circ}_{\overline{T} - \underline{T}}, \overline{p}{}^{\hat{\mathfrak{s}}_n}_{\overline{T} - \underline{T}}) \Big] &\leq \sqrt{\frac{1}{2}\E\Big[\KL{\overline{p}{}^{\mathfrak{s}^\circ}_{\overline{T} - \underline{T}}}{\overline{p}{}^{\hat{\mathfrak{s}}_n}_{\overline{T} - \underline{T}}} \Big]}\\
&= \sqrt{\E\Big[\int_{\underline{T}}^{\overline{T}} \int_{\overline{D}} f(x)\lvert \hat{\mathfrak{s}}_n(x,t) - \nabla \log p_t(x) \rvert^2 p_t(x) \diff{x} \diff{t}  \Big] }\\ 
&\leq \sqrt{\lVert f \rVert_{\overline{D}}\E\Big[\int_{\underline{T}}^{\overline{T}} \int_{\overline{D}}\lvert \hat{\mathfrak{s}}_n(x,t) - \nabla \log p_t(x) \rvert^2 p_t(x) \diff{x} \diff{t}  \Big] }\\
&\lesssim n^{-\frac{s}{2s+d}}(\log n)^3 (\log \log n)^{1/2},
\end{align*}
where we used Pinsker's inequality and Jensen's inequality together with concavity of $x \mapsto \sqrt{x}$ for the first line, while the second line follows from Theorem \ref{theo:girsa} (Girsanov), together with independence of the driving Brownian motion $\overline{W}$ and the initialisation $\overline{X}_0 \sim p_T$ from the data $(X_{0,i})_{i= 1,\ldots,n}$. 
The proof is concluded by applying the error decomposition \eqref{eq:error_decomp} and using that, by Lemma \ref{lem:early_stop} and Lemma \ref{lem:ergodic_err}, we have
\[\mathrm{TV}(p_0,p_{\underline{T}}) + \TV(\PP(X_{\overline{T}} \in \cdot \mid X_0 \sim p_0), \mathcal{U}(\overline{D})) \lesssim \underline{T}^{\beta/2} + \mathrm{e}^{-\lambda_1 \overline{T}} \lesssim n^{-\frac{s}{2s+d}},\]
by our choices of $\underline{T},\overline{T}$.
\end{proof}

\section{Discussion}\label{sec:discussion}
This paper presents a rigorous investigation of the non-standard class of denoising reflected diffusion models (DRDMs), focusing on statistical convergence guarantees and approximation within constrained settings, as it is relevant for scenarios involving bounded state spaces. This is a first step in extending the statistical analysis of standard denoising diffusion models to the generalised class of \textit{denoising Markov models}  proposed in \cite{benton_shi24}.

{\color{black}A key result of our analysis is the derivation of  convergence rates in total variation that match the minimax lower bound on Sobolev classes  up to a logarithmic factor. More precisely, from \citet[Theorem 4]{yang99}, see also \citet[Proposition 5.2]{oko23} for the verification of their assumptions for $s$-smooth Sobolev functions on $[0,1]^d$, we have that 
\[\inf_{\hat{p}_n} \sup_{p_0 \in H^s(D)} \E_{p_0}[\mathrm{TV}(p_0,\hat{p}_n)] \asymp n^{-\frac{s}{2s+d}},\] 
where the infimum ranges over all estimators $\hat{p}_n$ based on $n$ i.i.d.\ data points having density $p_0$ under $\PP_{p_0}$. Our upper bound stated in \eqref{eq:conv_rate} therefore establishes that DRDMs attain the minimax optimal rate of convergence up to $\log$-factors for specific Sobolev densities.} Note that our logarithmic loss is comparatively small relative to that of unconstrained DDMs in \cite{oko23}, where it is of order $(\log n)^8$.

However, convergence rates (even optimal ones) expressed in terms of the ambient dimension $d$ fall short of capturing the empirical success of DDMs. This gap is related to the \emph{manifold hypothesis}, a prominent idea that real-world high-dimensional data often reside on lower-dimensional manifolds, to which well-trained generative models are believed to adapt. Developing a theoretical underpinning for this hypothesis has therefore been one of the central goals in statistical theory for generative models. 
In the pioneering paper \cite{oko23}, the authors also take a first step towards investigating statistical convergence guarantees of DDMs for data distributed on such lower-dimensional structures by extending their analysis to initial distributions supported on a lower-dimensional hyperplane, where they obtain the almost minimax optimal rate $n^{-\frac{s+1}{2s + \tilde{d}} + \varepsilon}$ in the Wasserstein-1 distance in terms of the sample size $n$ and the subspace dimension $\tilde{d}$. Related statistical results under linear subspace assumptions are given in \cite{chen23b}. In the recent work \cite{tang24}, \citeauthor{tang24} significantly extend this result by establishing (up to $\log$ factors) the minimax convergence rate $C(d) n^{-\frac{s+1}{2s + \tilde{d}}}$ in Wasserstein-1 distance for distributions $p_0$ such that 
\begin{enumerate}[label = (\roman*)]
\item $p_0$ is supported on a compact and $\beta$-smooth $\tilde{d}$-dimensional submanifold $\mathcal{M}$ with positive reach, where $\beta \geq 2$;
\item $p_0$ is bounded away from zero on $\mathcal{M}$;
\item $p_0$ has smoothness of order $s \in [0,\beta-1]$ w.r.t.\ the volume measure on $\mathcal{M}$.
\end{enumerate}
Note that these three conditions mirror the three assumptions from \cite{oko23} mentioned in the introductory Section \ref{sec:intro} in the manifold setting. 
The multiplicative factor $C(d)$ in \cite{tang24}'s convergence rate is of order $d^{s +\tilde{d}/2}$ and thus potentially very large for high ambient dimension $d$.
Most recently, \cite{azangulov24} show that this multiplicative factor can be significantly reduced to the order $\sqrt{d}$ by appropriately choosing the neural network class for score approximation. Extending the DRDM framework to support data on submanifolds (and thus improving the convergence rate) presents additional mathematical challenges. For example, enforcing a reflecting boundary for data supported on lower-dimensional submanifolds would require non-trivial modifications to the spectral score representation and a revised analysis of the associated Sobolev bounds. Such adjustments are beyond the scope of this study, yet our current work may provide a foundational approach for future research in the reflected diffusion context.

The assumptions on $p_0$ in our model (see \ref{ass:init}) play a key role in controlling the approximation error in our DRDM framework under bounded domain constraints. Such assumptions, although somewhat limiting in a practical context, are comparable to those made in \cite{oko23} for the total variation convergence analysis of unconstrained models while maintaining compatibility with the spectral methods that we employ for our statistical analysis. In this context, our more stringent smoothness assumption $s > d/2$ implies Hölder continuity of the data density $p_0$, but allows us to avoid some technical difficulties arising due to the less explicit analytical nature of DRDMs compared to standard DDMs. 

{\color{black}
How to remove the lower bound assumption $p_0\vert_{\operatorname{supp} p_0} \geq \alpha$ that is present in all the works discussed above is a highly relevant and conceptually challenging question. Recent works by \cite{zhang24, stephano25}, and \cite{yakovlev25} prove minimax optimal rates for particular unconstrained diffusion models without lower bound assumptions for  sub-Gaussian densities $p_0$ and push-forward distributions on the ambient space $\R^D$ of the form $g_{\#} \mathcal{U}[0,1]^d$ for Hölder-continuous $g$ and $d \leq D$, respectively. \cite{zhang24} use a more classical kernel estimation and truncation strategy instead of neural network approximations. \cite{stephano25} exploit deeper results on the space-time regularity of the score function of an Ornstein--Uhlenbeck process for direct approximation with $\tanh$ activation function, thereby avoiding the need to approximate $p_t$ and $\nabla p_t$ separately. Finally \cite{yakovlev25} exploit the structure of the score induced by their data assumption $p_0 \sim g_{\#} \mathcal{U}[0,1]^d$  and the Gaussian Ornstein--Uhlenbeck forward transition densities to construct their ReLU neural network based approximation class in a very specific way. These approaches do not translate directly to our non-Gaussian reflected setting and we leave the statistical study of reflected diffusion models without lower bound assumptions to future work.}

In general, our analytical approach in DRDMs differs significantly from that in unconstrained diffusion models, as the bounded domain prevents the explicit Gaussian transition densities commonly used for error control. The semi-explicit nature of these densities in the reflected setting means that, rather than relying on straightforward Gaussian approximations, we implement general spectral decompositions and Sobolev-based bounds informed by the Sobolev smoothness of $p_0$. The score approximation here presents additional technical challenges, which we address through an innovative polynomial-time interpolation procedure that proves crucial to achieving feasible convergence rates. This technique introduces new and effective methods for controlling error contributions in generative models with bounded state spaces {\color{black} and, as outlined in the introduction, may  also provide a versatile tool for statistical analysis of generalised denoising Markov models \citep{benton_shi24, ren25} beyond the scope of this paper.}

Finally, it should be noted that our analysis does not address sampling issues, in particular those arising from simulating the backward reflected process with an estimated drift. While this aspect is important in practical implementations, our current focus on theoretical convergence guarantees serves to isolate and address the approximation errors inherent to the reflection-based model setup itself. Future studies could incorporate error analysis for sampling methodologies specific to reflected processes to extend the results presented here.

\appendix 
\section{Girsanov's theorem for reflected diffusions}
The following result is a version of Girsanov's theorem for reflected diffusions, which is a correction of Theorem 7.1 and Theorem A.6 in \cite{lou23}, {\color{black}where the influence of the diffusion coefficient on the KL divergence has been overlooked.}
\begin{theorem}\label{theo:girsa}
	\label{girsanov}
	Let $(X_t)_{t \in [0,T]}$ and $(\tilde{X}_t)_{t \in [0,T]}$ be solutions of the normally reflected SDEs 
	\begin{align} 
		\diff{X_t}
		&=b(t,X_t) \diff{t} + \sigma(t,X_t) \diff{W_t} + \nu(X_t)\diff{\ell_t},\notag \\ 
		\diff{\tilde{X}_t} 
		&= \tilde{b}(t,\tilde{X}_t) \diff{t} + \sigma(t,\tilde{X}_t) \diff{\tilde{W}_t} +  \nu(\tilde{X}_t)\diff{\tilde{\ell}_t}, \quad \tilde{X}_0 \overset{d}{=} X_0, \label{eq:girsa2}
	\end{align}
	where $b, \tilde{b}$ are bounded on $[0,T] \times \overline{D}$ and  $\sigma(t,\cdot)$ is bounded, Lipschitz and uniformly elliptic in the sense $a(t,\cdot) \coloneqq \sigma(t,\cdot)\sigma(t,\cdot)^\top \succeq \underline{\lambda} \mathbb{I}$ for some $\underline{\lambda} > 0$ and all $t \in [0,T]$. 
    Denote by $\mathbb{P}_{X^T}$ and $\mathbb{P}_{\tilde{X}^T}$ their respective path measures on $\mathcal{C}([0,T], \overline{D})$.
	Then, for $L_t = \int_0^{t} \nu(X_s) \diff{\ell_s}$,
	\begin{align*} 
		&\log \frac{\diff{\mathbb{P}_{\tilde{X}^T}}}{\diff{\mathbb{P}_{X^T}}}(X^T)\\ 
		&\,= \int_0^T (\tilde{b}(t,X_t) - b(t,X_t))^\top a^{-1}(t,X_t) \diff{(X_t-L_t)}\\
        &\qquad- \frac{1}{2} \int_0^T (\tilde{b}(t,X_t) - b(t,X_t))^\top a^{-1}(t,X_t)(\tilde{b}(t,X_t) + b(t,X_t)) \diff{t},
	\end{align*}
	a.s., and 
	\[
		\KL{\mathbb{P}_{X^T}}{\mathbb{P}_{\tilde{X}^T}}
		= \frac{1}{2} \mathbb{E}\Big[\int_0^T \lVert \sigma^{-1}(t,X_t) (\tilde{b}(t,X_t) - b(t,X_t)) \rVert^2 \diff{t}\Big].
	\]
\end{theorem}

\begin{proof} 
	Let 
	\begin{align*}
		Z_T
		&\coloneqq \exp\Big(\int_0^T (\tilde{b}(t,X_t) - b(t,X_t))^\top (\sigma^{-1}(t,X_t))^{\top}\diff{W_t}\\
        &\qquad - \frac{1}{2} \int_0^T \lVert \sigma^{-1}(t,X_t) (\tilde{b}(t,X_t) - b(t,X_t)) \rVert^2 \diff{t} \Big),
	\end{align*}
	and define $\diff{\mathbb{Q}_T} \coloneqq Z_T \diff{\mathbb{P}}$.
	Since $\overline{\beta} \coloneqq \sup_{t \in [0,T], x \in \overline{D}} \lVert \beta(t,x) \rVert < \infty$ for $\beta \in \{b , \tilde{b}\}$, we have 
	\[
		\sup_{t \in [0,T], x \in \overline{D}}\n{\sigma^{-1}(t,x) \beta(t,x)}
		\le\frac{\overline{\beta}}{\underline{\lambda}}  
		<\infty.
	\]
	Thus, Novikov's condition is satisfied, making $\mathbb{Q}_T$ a probability measure equivalent to $\PP$, and Girsanov's theorem implies that 
	\[
		\tilde{W}_t
		\coloneqq W_t - \int_0^t \sigma^{-1}(s,X_s)(\tilde{b}(s,X_s) - b(s,X_s)) \diff{s}, \quad t \in [0,T], 
	\]
	is a $\mathbb{Q}_T$-Brownian motion, see \citet[Chapter 3, Theorem 5.1, Corollary 5.13]{karatzas91}. 
	Thus,
	\[
		X_t
		= X_0 + \int_0^t \tilde{b}(s,X_s) \diff{s} + \int_0^t \sigma(s,X_s) \diff{\tilde{W}_s} + \int_0^t \nu(X_s) \diff{\ell_s}, \quad t \in [0,T],
	\]
	where $\ell_t  = \int_0^t \bm{1}_{\{X_s \in \partial D\}} \diff{\ell_s},$ for any $t \in [0,T]$, $\PP$-a.s.\ and hence also $\mathbb{Q}_T$-a.s. Since $X_0 \overset{d}{=} \tilde{X}_0$, it follows that $X^T$ is a weak solution to the reflected SDE \eqref{eq:girsa2} and analogously to \citet[Chapter 5, Proposition 3.10]{karatzas91}, it holds that under the given assumptions on $\tilde{b},\sigma$, weak solutions to \eqref{eq:girsa2} are unique in law. 
	Thus, $X^T$ under $\mathbb{Q}_T$ has the same law as $\tilde{X}^T$ under $\mathbb{P}$, such that, upon noting that
	\begin{align*}
		\log Z_T
		& = \int_0^T (\tilde{b}(t,X_t) - b(t,X_t))^\top a^{-1}(t,X_t) \diff{(X_t - L_t)}\\
        &\qquad - \frac{1}{2} \int_0^T  (\tilde{b}(t,X_t) - b(t,X_t))^\top a^{-1}(t,X_t)(\tilde{b}(t,X_t) + b(t,X_t)) \diff{t},
	\end{align*}
	the first assertion follows.
	This yields immediately
	\begin{align*} 
		\KL{\mathbb{P}_{X^T}}{\mathbb{P}_{\tilde{X}^T}}
		=-\mathbb{E}\Big[\log \frac{\diff{\mathbb{P}_{\tilde{X}^T}}}{\diff{\mathbb{P}_{X^T}}}(X^T) \Big]
		&=-\mathbb{E}[\log Z_T]\\
		&=\frac{1}{2}\mathbb{E}\Big[\int_0^T \n{\sigma^{-1}(t,X_t) (\tilde{b}(t,X_t) - b(t,X_t))}^2 \diff{t}\Big].
	\end{align*}
\end{proof}

\section{Bernstein condition for the denoising score matching excess loss and generalisation error}\label{app:bernstein}
The goal of this section is to prove \eqref{eq:bernstein_condition} and thereby justify the generalisation error decomposition from Theorem \ref{theo:score_loss_emp} in a generalised Markovian framework that in particular includes our reflected forward model, as well as the unconstrained Ornstein--Uhlenbeck models used in other works. 

Let $Y_1,\ldots, Y_n \overset{\mathrm{i.i.d}}{\sim} \mu$ be a collection of random variables with state space $\mathcal{X} \subset \R^d$. Let also $(X_t)_{t \geq 0}$ be a Markov process with state space $\mathcal{X}$ and transition densities $(q_t)_{t \geq 0}$ w.r.t.\ some reference measure $\nu$, that is, $\PP^x(X_t \in \diff{y}) = q_t(x,y) \,\nu(\diff{y})$, where $\PP^x = \PP(\cdot \mid X_0 = x)$. Denote by $p_t$ the density of $X_t$ started in $\mu$ at time $t > 0$, i.e., $p_t(y) \nu(\diff{y}) = \PP^\mu(X_t \in \diff{y}) = \int_{\mathcal{X}} q_t(x,y) \, \mu(\diff{x})\, \nu(\diff{y})$. We assume that the \textit{score} $s^\circ$ of this Markov process given by $s^\circ(x,t) = \nabla \log p_t(x)$ as well as $\nabla \log q_t(x,y)$ are well defined for all $t > 0, x,y \in \mathcal{X}$, and define for some measurable candidate function $s$ the \textit{denoising score matching loss} over an interval $[\underline{T}, \overline{T}] \subset (0,\infty)$ by 
\begin{align*} 
L_s(x) &= \int_{\underline{T}}^{\overline{T}} \int_{\mathcal{X}} \lvert s(y,t) - \nabla \log q_t(x,y) \rvert^2 q_t(x,y) \,\nu(\diff{y}) \diff{t}\\
&= \E^x\Big[\int_{\underline{T}}^{\overline{T}} \lvert s(X_t,t) - \nabla \log q_t(x,X_t) \rvert^2 \diff{t} \Big],
\end{align*}
which gives
\[\E^\mu[L_s(X_0)] = \int_{\underline{T}}^{\overline{T}} \iint_{\mathcal{X} \times \mathcal{X}} \lvert s(y,t) - \nabla \log q_t(x,y) \rVert^2 q_t(x,y) \,\nu(\diff{y})\, \mu(\diff{x}) \diff{t} .\]
Then, given sufficient integrability properties that allow interchanging the order of differentiation and integration, a few simple lines of calculation establish the equivalence between denoising and \textit{explicit} score matching \cite{vincent11}, which is the first equality in 
\begin{align*}
\E^\mu[L_s(X_0)] &= \int_{\underline{T}}^{\overline{T}} \int_{\mathcal{X}} \lvert s(y,t) - \nabla \log p_t(y) \rvert^2 p_t(y) \,\nu(\diff{y})  \diff{t} + C_{\underline{T},\overline{T}}\\
&= \int_{\underline{T}}^{\overline{T}} \int_{\mathcal{X}} \lvert s(y,t) - s^\circ(y,t) \rvert^2 p_t(y) \,\nu(\diff{y})  \diff{t} + C_{\underline{T},\overline{T}},
\end{align*}
where $C_{\underline{T},\overline{T}} = \E^\mu[L_{s^\circ}(X_0)]$ is independent of $\mathcal{S}$. In particular, we obtain 
\begin{equation}\label{eq:equiv_denoise}
\int_{\underline{T}}^{\overline{T}} \int_{\mathcal{X}} \lvert s(y,t) - \nabla \log p_t(y) \rvert^2 p_t(y) \,\nu(\diff{y})  \diff{t} = \E^\mu[L_s(X_0)] - \E^\mu[L_{s^\circ}(X_0)].
\end{equation}
As discussed before, a natural and tractable choice for the score $s^\circ$ given the data $Y_1,\ldots,Y_n$ is therefore given by the empirical risk minimiser 
\[\hat{s} \in \argmin_{s \in \mathcal{S}} \frac{1}{n} \sum_{i=1}^n L_s(Y_i),\]
where $\mathcal{S}$ is some appropriate class of candidate functions. 
The following result shows that the excess risk $L_s - L_{s^\circ}$ satisfies a Bernstein condition w.r.t.\ $\mu$, which is essential for the denoising score loss to work well as an empirical risk minimisation objective.

\begin{lemma} 
Suppose that $\sup_{s \in \mathcal{S} \cup \{s^\circ\}} \lVert L_{s} \rVert_{\mathcal{X}} \leq C(\mathcal{L}) < \infty$. Then, 
\[\E^\mu\big[(L_s(X_0) - L_{s^\circ}(X_0))^2\big] \leq 4C(\mathcal{L}) \E^\mu[L_s(X_0) - L_{s^\circ}(X_0)].\]
\end{lemma}
\begin{proof} 
Using the elementary equality $\lvert a \rvert^2 - \lvert b \rvert^2 = \langle a+b,a-b\rangle$ for vectors $a,b$, and the Cauchy--Schwarz inequality several times, we find for any $x \in \cX$,
\begin{align*} 
&\lvert L_s(x) - L_{s^\circ}(x)\rvert\\
&\,\leq \int_{\underline{T}}^{\overline{T}} \int_{\mathcal{X}} \lvert \langle s(y,t) - s^\circ(y,t), s(y,t) + s^\circ(y,t) - 2 \nabla \log q_t(x,y)\rangle \rvert q_{t}(x,y) \,\nu(\diff{y}) \diff{t}\\ 
&\,\leq \int_{\underline{T}}^{\overline{T}} \int_{\mathcal{X}} \lvert s(y,t) - s^\circ(y,t) \rvert \lvert s(y,t) + s^\circ(y,t) - 2 \nabla \log q_t(x,y)\rvert  q_{t}(x,y) \,\nu(\diff{y}) \diff{t}\\ 
&\,\leq \int_{\underline{T}}^{\overline{T}} \Big(\int_{\mathcal{X}} \lvert s(y,t) - s^\circ(y,t) \rvert^2 q_t(x,y)\,\nu(\diff{y})\Big)^{1/2} \\
&\,\qquad \quad \times\Big(\int_{\mathcal{X}} \lvert s(y,t) + s^\circ(y,t) - 2 \nabla \log q_t(x,y)\rvert^2  q_{t}(x,y) \,\nu(\diff{y})\Big)^{1/2} \diff{t} \\ 
&\,\leq \Big(\int_{\underline{T}}^{\overline{T}} \int_{\mathcal{X}} \lvert s(y,t) - s^\circ(y,t) \rvert^2 q_t(x,y)\,\nu(\diff{y}) \diff{t}\Big)^{1/2} \\ 
&\, \qquad\quad\times\Big(\int_{\underline{T}}^{\overline{T}} \int_{\mathcal{X}} \lvert s(y,t) + s^\circ(y,t) - 2 \nabla \log q_t(x,y)\rvert^2  q_{t}(x,y) \,\nu(\diff{y}) \diff{t} \Big)^{1/2} \\
&\,\leq \sqrt{2(L_s(x) + L_{s^\circ}(x))} \Big(\int_{\underline{T}}^{\overline{T}} \int_{\mathcal{X}} \lvert s(y,t) - s^\circ(y,t) \rvert^2 q_t(x,y)\,\nu(\diff{y}) \diff{t}\Big)^{1/2}\\
&\,\leq 2\sqrt{C(\mathcal{L})} \Big(\int_{\underline{T}}^{\overline{T}} \int_{\mathcal{X}} \lvert s(y,t) - s^\circ(y,t) \rvert^2 q_t(x,y)\,\nu(\diff{y}) \diff{t}\Big)^{1/2} 
\end{align*}
Thus, using \eqref{eq:equiv_denoise},
\begin{align*} 
\E^\mu\big[(L_s(X_0) - L_{s^\circ}(X_0))^2\big] &\leq 4C(\mathcal{L}) \int_{\mathcal{X}}\int_{\underline{T}}^{\overline{T}} \int_{\mathcal{X}} \lvert s(y,t) - s^\circ(y,t) \rvert^2 q_t(x,y)\,\nu(\diff{y}) \diff{t} \,\mu(\diff{x}) \\
&= 4C(\mathcal{L}) \int_{\underline{T}}^{\overline{T}} \int_{\mathcal{X}} \lvert s(y,t) - s^\circ(y,t) \rvert^2 p_t(y) \,\nu(\diff{y}) \diff{t} \\ 
&= 4C(\mathcal{L}) \E^\mu[L_s(X_0) - L_{s^\circ}(X_0)].
\end{align*} 
\end{proof}

 This result renders the reasoning in \citet[Theorem C.4]{oko23} valid up to a multiplicative factor that only results in a minor change of  constants in their generalisation error upper bound. For completeness we state a corrected version of \citet[Theorem C.4]{oko23} in the general context of this section with corrected constants. 
 
\begin{theorem}
	Suppose that $\sup_{s \in \mathcal{S}} \lVert L_{s} \rVert_{\mathcal{X}} \leq C(\mathcal{L})$, where $C(\mathcal{L}) < \infty$.  Then, for any $\delta > 0$ such that $\mathcal{N}(\mathcal{L},\lVert \cdot \rVert_{\mathcal{X}}, \delta) \geq 3$, it holds that
	\begin{equation*} 
		\begin{split}
			&\E\Big[\int_{\underline{T}}^{\overline{T}} \int_{\mathcal{X}} \lvert \hat{s}(x,t) - \nabla \log p_t(x) \rvert^2 p_t(x) \,\nu(\diff{x}) \diff{t}  \Big] \\ 
			&\, \leq 2 \inf_{s \in \mathcal{S}} \int_{\underline{T}}^{\overline{T}} \int_{\mathcal{X}} \lvert s(x,t) - \nabla \log p_t(x) \rvert^2 p_t(x) \,\nu(\diff{x}) \diff{t} + \frac{2C(\mathcal{L})}{n}\Big(\frac{145}{9}\log \mathcal{N}(\mathcal{L}, \lVert \cdot \rVert_{\cX}, \delta) + 160\Big)\\
            &\qquad+ 5\delta.
		\end{split}
	\end{equation*}
\end{theorem}

\section{Technical results on neural network approximations}\label{app:neural}
We first consider neural network approximations of products.
\begin{lemma}
\label{lemma:mult_network}
    For any $C\ge1$ and $m, d\in\N$, there exist neural networks $\varphi^{\mathrm{mult}}_m\in\Phi(L, W, S, B)$ and $\varphi^{\mathrm{mult}, d}_m\in\Phi(L, W_d, d\cdot S, B)$ satisfying
    \[
        |\varphi^{\mathrm{mult}}_m(x, y)-xy|
        \le C2^{-m},\quad x\in[0, 1],y\in[-C, C].
    \]
    and
    \[
        |\varphi^{\mathrm{mult}, d}_m(x, y)-xy|
        \le\sqrt{d}C2^{-m},\quad x\in[0, 1],y\in[-C, C]^d.
    \]
    Furthermore, $\varphi^{\mathrm{mult}}(0, y)=\varphi^{\mathrm{mult}}(x, 0)=0$.
    The sizes of the networks are evaluated as $L=m+8$, $S=58+16m$, $B=C$ and
    \begin{gather*}
        W
        =\mat{2 & 3 & 3 & \smash[b]{\underbrace{\begin{matrix}
            12 & \cdots & 12
        \end{matrix}}_{m+2\text{ times}}} & 2 & 2 & 1}^\top,\quad
        (W_d)_i
        =\begin{cases}
            d+1, &\text{ if }i=1\\
            d\cdot W_i &\text{ otherwise.}
        \end{cases}
    \end{gather*}
    In particular, we have $L, S\lesssim m$, $\n{W}_{\infty}\lesssim1$, $\n{W_d}_{\infty}\lesssim d$ and $B\lesssim C$.
\end{lemma}

\begin{proof}
    We proceed as in the proof of \citet[Lemma F.6]{oko23}, but adapted to this specific setting.
    Thus, let $m\in\N$ and $C\ge0$ be fixed.
    Then, by \citet[Lemma A.2]{schmidthieber20}, there exists a neural network $\overline{\varphi}^{\mathrm{mult}}_m\in\Phi(m+4, W_0, 24+16m, 1)$, where
    \begin{gather*}
        W_0
        =\mat{2 & \smash[b]{\underbrace{\begin{matrix}
            6 & \cdots & 6
        \end{matrix}}_{m+2\text{ times}}} & 1}^\top, \\
    \end{gather*}
    satisfying $|\overline{\varphi}^{\mathrm{mult}}_m(x, y)-xy|\le 2^{-m}$ and $\overline{\varphi}^{\mathrm{mult}}_m(x, 0)=\overline{\varphi}^{\mathrm{mult}}_m(0, y)=0$ for all $x, y\in[0, 1]$.
    Thus, for $y\in[-1, 1]$, we have $|\mathrm{sgn}(y)\overline{\varphi}^{\mathrm{mult}}_m(x, |y|)-xy|\le 2^{-m}$.
    Note then that
    \begin{align*}
        \mathrm{sgn}(y)\overline{\varphi}^{\mathrm{mult}}_m(x, |y|)
        &=\sigma_R\big(\overline{\varphi}^{\mathrm{mult}}_m(x, \sigma_R(y))\big)-\sigma_R\big(\overline{\varphi}^{\mathrm{mult}}_m(x, \sigma_R(-y))\big) \\
        &=\overline{\varphi}^{{\mathrm{mult}}, 3}_m\circ\overline{\varphi}^{{\mathrm{mult}}, 2}_m\circ\overline{\varphi}^{{\mathrm{mult}}, 1}_m(x, y),
    \end{align*}
    where
    \begin{alignat*}{2}
        \overline{\varphi}^{{\mathrm{mult}}, 1}_m
        &=I_3\sigma\mat{1 & 0 \\ 0 & 1 \\ 0 & -1}
        &&\in\Phi(1, \mat{2 & 3 & 3}^\top, 6, 1), \\
        \overline{\varphi}^{{\mathrm{mult}}, 3}_m
        &=\mat{1 & -1}\sigma I_2
        &&\in\Phi(1, \mat{2 & 2 & 1}^\top, 4, 1),
    \end{alignat*}
    and $\overline{\varphi}^{{\mathrm{mult}}, 2}_m(x, y, z)=\mat{\overline{\varphi}^{\mathrm{mult}}_m(x, y) & \overline{\varphi}^{\mathrm{mult}}_m(x, z)}^\top$.
    Thus, $\overline{\varphi}^{{\mathrm{mult}}, 2}_m\in\Phi(m+4, W_1, 48+32m, 1)$ where $W_1$ is the parallelization of $W_0$ with itself, and hence $\mathrm{sgn}(y)\overline{\varphi}^{\mathrm{mult}}_m(x, |y|)\coloneqq\widetilde{\varphi}^{\mathrm{mult}}_m(x, y)\in\Phi(m+8, W, 58+16m, 1)$, where $W$ is as specified in the statement of the lemma.
    Finally, if $y\in[-C, C]$, we have similarly $|C\widetilde{\varphi}_m^{\mathrm{mult}}(x, y/C)-xy|\le C2^{-m}$, and since these are simply linear transformations (i.e., neural networks of depth $0$), it follows that $C\widetilde{\varphi}_m^{\mathrm{mult}}(x, y/C)\coloneqq\varphi_m^{\mathrm{mult}}(x, y)\in\Phi(m+8, W, 58+16m, C)$.
    Parallelization of $\varphi_m^{\mathrm{mult}}$ a total of $d$ times yields a network $\varphi_m^{{\mathrm{mult}}, d}\in\Phi(m+8, W_d, d(58+16m), C)$ with $W_d$ as specified above, which satisfies
    \[
        \max_{i\in\{1, \ldots, d\}}|(\varphi_m^{{\mathrm{mult}}, d}(x, y))_i-xy_i|\le C2^{-m},\quad x\in[0, 1],y\in[-C, C]^d,
    \]
    implying the desired result since $|y|\le\sqrt{d}\n{y}_\infty$ for $y\in\R^d$.
    \end{proof}

Next, we consider the reciprocals.
Note that by not relying on the comparatively slow converging Taylor series of the reciprocal, we achieve a substantially smaller network size than \citet[Lemma F.7]{oko23}, who achieve a network size of $\widetilde{\Phi}(m^2, m^3, m^4, 2^{2m})$ (assuming, as they do, that $m=\underline{k}=\overline{k}$).

\begin{lemma}
    For any $\underline{k}, \overline{k}, m\in\N$ with $m>\overline{k}$, there exists a neural network $\varphi_m^\mathrm{rec}\in\Phi(L, W, S, B)$ satisfying
    \[
        |\varphi_m^{\mathrm{rec}}(x)-\frac{1}{x}|\le2^{-m},\quad x\in[2^{-\underline{k}}, 2^{\overline{k}}].
    \]
    The size of $\varphi_m^{\mathrm{rec}}$ is evaluated as
    \begin{align*}
        L
        &=\big(4(\underline{k}+\overline{k})+2m+17\big)\lceil\log_2(\underline{k}+m+2)\rceil
        +2\lceil\log_2(\underline{k}+\overline{k})\rceil+1, \\
        S
        &=\big(260+68(\underline{k}+\overline{k}+m)\big)\lceil\log_2(\underline{k}+m+2)\rceil+8(\underline{k}+\overline{k}),
    \end{align*}
    $B=2^{2(\underline{k}+\overline{k})}$ and $W$ is the concatenation of $\lceil\log_2(\underline{k}+m+2)\rceil$ copies of $W_{\mathrm{iter}}$ and $W_{\mathrm{init}}$, where
    \begin{gather*}
        W_\mathrm{iter}
        =\mat{2 & \smash[b]{\underbrace{\begin{matrix}
            7 & \cdots & 7
        \end{matrix}}_{2(\underline{k}+\overline{k})+m+3\text{ times}}} & 2 & \smash[b]{\underbrace{\begin{matrix}
            7 & \cdots & 7
        \end{matrix}}_{2(\underline{k}+\overline{k})+m+3\text{ times}}} & 2}^\top, \\
    \end{gather*}
    and
    \[
        W_{\mathrm{init}}
        =\mat{
            1 & \underline{k}+\overline{k} & 2^{\lceil\log_2(\underline{k}+\overline{k})\rceil} & 2^{\lceil\log_2(\underline{k}+\overline{k})\rceil-1} & 2^{\lceil\log_2(\underline{k}+\overline{k})\rceil-1} & \cdots & 2 & 2 & 1
        }^\top.
    \]
    In particular, we have $L, S\lesssim (k+m)(\log(k+m))$, $\n{W}_\infty\lesssim k$, where $k=\underline{k}+\overline{k}$.
\end{lemma}

\begin{proof}
    Let $\underline{k}, \overline{k}, m\in\N$ with $m>\overline{k}$ be fixed.
    We approximate the reciprocal by Newton--Raphson iterations, adapted to neural networks.
    In a usual Newton--Raphson scheme approximating $x^{-1}$ for some $x>0$, we would take $x_0$ to be some initial approximation and set $x_n=x_{n-1}(2-xx_{n-1})$ for $n\in\N$.
    This, however, involves two multiplications by non-constants in each iteration, and as such is not directly accessible by neural networks.
    To overcome this, suppose that we have access to a neural network $\varphi_m^{\mathrm{iter}}$ such that $|\varphi_{m}^{\mathrm{iter}}(x, y)-x(2-xy)|\le2^{-(m+1)}$ as well as a neural network $\varphi^{\mathrm{init}}$ such that $|\varphi^{\mathrm{init}}-x^{-1}|\le\frac{1}{2}x^{-1}$.
    Then, setting $\widetilde{x}_0=\varphi^{\mathrm{init}}(x)$ and $\widetilde{x}_{n}=\varphi^{\mathrm{iter}}_m(\widetilde{x}_{n-1}, x)$ for $n\in\N$, we claim that $|\widetilde{x}_{n_0}-x^{-1}|\le 2^{-m}$ where $n_0=\lceil\log_2(\underline{k}+m+2)\rceil$.
    To show this, we first find by the recursive definition of $\widetilde{x}_n$ that
    \[
        |\widetilde{x}_{n+1}-x^{-1}|
        =|\varphi_m^{\mathrm{iter}}(\widetilde{x}_n, x)-x^{-1}|
        \le|\widetilde{x}_n(2-x\widetilde{x}_n)-x^{-1}|+2^{-(m+1)}
        =x(\widetilde{x}_n-x^{-1})^2+2^{-(m+1)}.
    \]
    Hence, setting $e_n(x)\coloneqq xe_{n-1}^2(x)+2^{-(m+1)}$ with $e_0(x)\coloneqq\frac{1}{2}x^{-1}$, we find that $|\widetilde{x}_n-x^{-1}|\le e_n(x)$ for all $n\in\N$.
    To show that $e_{n_0}(x)\le 2^{-m}$ on $[2^{-\underline{k}}, 2^{\overline{k}}]$, we introduce $d_n(x)\coloneqq e_n(x)-\frac{1}{x}2^{-2^n}$ for $n\in\N$ and $x\in[2^{-\underline{k}}, 2^{\overline{k}}]$. Since $\frac{1}{x}2^{-2^n}=x\big(\frac{1}{x}2^{-2^{n-1}}\big)^2$, $d_n$ satisfies the recursion $d_n(x)=xd_{n-1}^2(x)+2^{-2^n+1}d_{n-1}(x)+2^{-(m+1)}$ with $d_0\equiv0$.
    Furthermore, we claim that $d_n$ is non-decreasing and convex.
    Clearly, this is true for $d_0$, and by the recursion above we have
    \[
        d_n'(x)
        =\frac{\mathrm{d}}{\mathrm{d}x}[x d_{n-1}^{2}(x)+2^{-2^n+1}d_{n-1}(x)]
        =d_{n-1}^2(x)+2xd_{n-1}'(x)d_{n-1}(x)+2^{-2^n+1}d_{n-1}'(x).
    \]
    By induction, this is non-negative and non-decreasing, so $d_n$ is non-decreasing and convex as claimed.
    Since $e_n(x)=\frac{1}{x}2^{-2^n}+d_n(x)$, it follows that $e_n$ is also convex, and we thus only need to check that $e_{n_0}(2^{-\underline{k}})\vee e_{n_0}(2^{\overline{k}})\le 2^{-m}$.
    To this end, it is clear that for $n\ge2$, we have $d_n(2^{-\underline{k}})\ge d_{n+1}(2^{-\underline{k}})$, and some tedious but straightforward calculations show that
    \[
     	d_3(2^{-\underline{k}})
     	=\Big(1+2^{-\underline{k}}2^{-(m+1)}\Big(2^{-k}2^{-(m+1)}+\frac{17}{16}\Big)^2+\frac{1}{256}\Big(2^{-k}2^{-(m+1)}+\frac{17}{16}\Big)\Big)2^{-(m+1)}.
    \]
    A very rough estimate yields from this that $d_n(2^{-k})\le\frac{3}{2}2^{-(m+1)}=2^{-m}-2^{-(m+2)}$ for $n\ge3$.
	Now, since $\underline{k}+m+2>4$, we have $n_0\ge 3$, and so
    \[
        e_{n_0}(2^{-\underline{k}})
        =2^{\underline{k}}2^{-2^n}+d_{n_0}(2^{-\underline{k}})
        \le2^{\underline{k}}2^{-(\underline{k}+m+2)}+2^{-m}-2^{-(m+2)}
        =2^{-m}.
    \]
    As for $e_{n_0}(2^{\overline{k}})$, since $\overline{k}\le m-1$, the same convexity argument implies that it suffices to check that $e_{n_0}(2^{m-1})\le 2^{-m}$.
    But we clearly have that $e_0(2^{m-1})=2^{-m}$, and since
    \[
        2^{m-1}(2^{-m})^2+2^{-(m+1)}
        =2\cdot 2^{-(m+1)}
        =2^{-m},
    \]
    it follows that $e_n(2^{m-1})=2^{-m}$ for all $n\in\N$.
    Now, let us construct this network.
    For notation, let $k=\underline{k}+\overline{k}$.
    We begin by constructing $\varphi^{\mathrm{init}}$ as a piecewise linear function.
    In particular, for $i=0,\ldots,k$, let $p_i=2^{i-\underline{k}}$, and for $i=1,\ldots,k$ let $\ell_i$ be the linear interpolation between $(p_{i-1}, p_{i-1}^{-1})$ and $(p_i, p_i^{-1})$, i.e.,
    \[
        \ell_i(x)
        =\frac{p_i^{-1}-p_{i-1}^{-1}}{p_i-p_{i-1}}x-\Big(\frac{p_i^{-1}-p_{i-1}^{-1}}{p_i-p_{i-1}}p_i-\frac{1}{p_i}\Big)
        =\frac{1}{p_i}\big(3-\frac{2}{p_i}x\big),\quad x\in[p_{i-1}, p_i].
    \]
    Note that by convexity of the reciprocal function, we have $\ell_i(x)\ge x^{-1}$ for all $x\in[p_{i-1}, p_i]$, and hence $E_i(x)\coloneqq|\ell_i(x)-\frac{1}{x}|=\frac{1}{p_i}\big(3-\frac{2}{p_i}x\big)-\frac{1}{x}$.
    We then have that $E_i'(x)=\frac{1}{x^2}-\frac{2}{p_i^2}$, and since $E_i$ has exactly one extremal point and this is a maximum, we have that the error is maximized when $x=\frac{p_i}{\sqrt{2}}$ and that the corresponding error is $E_i(\frac{p_i}{\sqrt{2}})=\frac{3-2\sqrt{2}}{p_i}\le\frac{1}{2x}$.
    Thus, since $\ell_i$ is merely an affine transformation and hence clearly representable as a neural network, $\ell_i$ satisfies the properties of $\varphi^{\mathrm{init}}$ on $[p_{i-1}, p_i]$.
    Also, if $x\in[p_{i-1}, p_i]$, then $\ell_j(x)<x^{-1}$ for all $j\neq i$, and so $\varphi^{\mathrm{init}}(x)\coloneqq\max_{j=1,\ldots,k}\ell_j(x)$ works.
    Since $a\vee b=a+\sigma(b-a)$ for $a, b\ge0$, the function $(y_1, \ldots, y_{k})\mapsto\max_{j=1,\ldots,k}y_j$ is realized by a network $\varphi^{\max}\in\Phi(2n_{k}, W_{\max}, 5k, 1)$, where $n_{k}=\lceil\log_2(k)\rceil$ and
    \[
        W_{\max}
        =\mat{k & 2^{n_{k}} & 2^{n_{k}-1} & 2^{n_{k}-1} & \cdots 2 & 2 & 1}^\top.
    \]
    Thus, by parallelizing, we have $\varphi^{\mathrm{init}}\in\Phi(2(n_{k}+1), \mat{1 & k & W_{\max}^\top}^\top, 8k, 2^{2k})$.
    Next, we construct $\varphi^{\mathrm{iter}}_m$.
    Similarly to the proof of Lemma \ref{lemma:mult_network}, a small modification of \citet[Lemma A.2]{schmidthieber20} yields a network $\varphi^{\mathrm{mult}}_l\in\Phi(l+4, W_{\mathrm{mult}}, 24+16l, 2^{k})$ with
    \begin{gather*}
        W_\mathrm{mult}
        =\mat{2 & \smash[b]{\underbrace{\begin{matrix}
            6 & \cdots & 6
        \end{matrix}}_{l+2\text{ times}}} & 1}^\top, \\
    \end{gather*}
    satisfying $|\varphi^\mathrm{mult}_l(x, y)-xy|\le2^{k+1-l}$ for all $x\in[0, 2]$ and $y\in[0, 2^{k}]$.
    Since we always have $x\wedge \widetilde{x}_n\le 2$ and $|2-x\widetilde{x}_n|\le 2$, setting $\varphi^{\mathrm{iter}}_m(x, y)=\varphi^{\mathrm{mult}}_l(x, 2-\varphi^{\mathrm{mult}}_l(x, y))$ with $l\ge 2k+m+3$ works.
    Parallelizing with the identity, we can now simply chain these together without altering the size of the network other than increasing the width of each layer by $1$.
    That is, setting $\widetilde{\varphi}^{\mathrm{init}}(x)\coloneqq\mat{\varphi^{\mathrm{init}}(x) & x}^\top$ and $\widetilde{\varphi}^{\mathrm{iter}}_m(x, y)\coloneqq\mat{\varphi^{\mathrm{iter}}_m(x, y) & x}^\top$,
    \[
        \varphi_m^{\mathrm{rec}}(x)
        \coloneqq\varphi^{\mathrm{iter}}_m\circ\underbrace{\widetilde{\varphi}^{\mathrm{iter}}_m\circ\cdots\circ\widetilde{\varphi}^{\mathrm{iter}}_m}_{n_0-1\text{ times}}\circ\widetilde{\varphi}^{\mathrm{init}}(x),
    \] 
    yields the desired network.
\end{proof}

\begin{proof}[Proof of Lemma \ref{lem:cap_network}]
	We first approximate $t\mapsto\sqrt{t}$ by a piecewise linear function.
	To this end, let $t_i=2^{-m}+i/m$ for $i=0,\ldots,m$ such that $t_m\ge1$.
	Next, for $j=1,\ldots,m$ and $t\in[t_{j-1}, t_j]$, set
	\[
		\ell_j(t)
		=m(\sqrt{t_j}-\sqrt{t_{j-1}})(t-t_{j-1})+\sqrt{t_{j-1}},
	\]
	such that $\ell_j$ is the linear interpolation between the points $(t_{j-1}, \sqrt{t_{j-1}})$ and $(t_j, \sqrt{t_j})$.
	Since $t\mapsto\sqrt{t}$ is concave, it follows that, for $t\in[t_{i-1}, t_i]$, we have $\ell_i(t)\le\sqrt{t}$ while $\ell_j(t)\ge\sqrt{t}$ for all $j\neq i$.
	Hence, setting $\ell(t)=\min_{j\in[m]}\ell_j(t)$ yields the desired piecewise linear approximation, and since both affine functions and minima are exactly representable as neural networks, we find that $\ell\in\widetilde{\Phi}(\log m, m, m, \sqrt{m})$.
	We claim that setting $\varphi^{\mathrm{cap}}=\varphi^{\mathrm{rec}}_1\circ\ell$ yields the desired network.
	Indeed, for the size of the network, note that since $\ell(t)\in[2^{-m/2}, 1]$ for all $t\in[2^{-m}, 1]$, we have $\varphi^{\mathrm{rec}}_{1}\in\widetilde{\Phi}(m\log m, m, m\log m, 2^{m/2})$, and since this dominates the size of $\ell$, it follows that $\varphi^{\mathrm{cap}}$ is of the same size as $\varphi_1^{\mathrm{rec}}$.

	As for the claim that $\varphi^{\mathrm{cap}}(t)\asymp\frac{1}{\sqrt{t}}$, first note that
	\begin{align*}
		\varphi^{\mathrm{cap}}(t)
		=\frac{1}{\sqrt{t}}
		+\Big(\varphi^{\mathrm{rec}}_1\circ\ell(t)-\frac{1}{\ell(t)}\Big)
		+\Big(\frac{1}{\ell(t)}-\frac{1}{\sqrt{t}}\Big),
	\end{align*}
	where the second term is in $[-1/2, 1/2]$, whence
	\[
		\frac{1}{2\sqrt{t}}
		\le\frac{1}{\sqrt{t}}+\Big(\varphi^{\mathrm{rec}}_1\circ\ell(t)-\frac{1}{\ell(t)}\Big)
		\le\frac{3}{2\sqrt{t}}.
	\]
	For the last term, note that since $\ell(t_i)=\sqrt{t_i}$ for $i=0,\ldots,m$, we have
	\[
		\Big|\frac{1}{\ell(t)}-\frac{1}{\sqrt{t}}\Big|
		\le\frac{1}{2^{m/2}}-\frac{1}{\sqrt{2^m+m}}
		<\frac{1}{4},
	\]
	whereby
	\[
		\frac{1}{4\sqrt{t}}
		\le\varphi^{\mathrm{cap}}(t)
		\le\frac{7}{4\sqrt{t}},
	\]
	as desired.
\end{proof}

\paragraph{Acknowledgement}
Financial support from the Villum Synergy Programme, Project
No.\ 50099, is gratefully acknowledged.
\printbibliography

\end{document}